\long\def\skipit#1{} 
\def\={\,=\,}
\def\+{\,+\,}
\def\-{\,-\,}

\def\dif{{\mathord{{\rm d}}}}
\def\mR{{\mathbb R}}
\def\mP{{\mathbb P}}

\newcommand{\eu}{{\mathfrak e}}  
\newcommand{\Eu}{{\mathscr{E}}}  

\documentclass[12pt]{amsart}
\usepackage{pstricks-add}
\usepackage{amsfonts}
\usepackage{amsmath}
\usepackage{amssymb}
\usepackage{ifthen,calc}
\usepackage{color}
\usepackage{epsfig}
\usepackage{graphicx}
\usepackage{epstopdf}
\usepackage{epsfig}
\usepackage{latexsym}
\usepackage{multicol}
\usepackage[sc]{mathpazo}
\usepackage[format=hang, margin=10pt]{caption}
\usepackage[mathscr]{eucal}

\newcounter{hours}
\newcounter{minutes}
\newcommand{\printtime}{
	\setcounter{hours}{\time/60}%
	\setcounter{minutes}{\time-\value{hours}*60}
	\ifthenelse{\value{hours}<10}{0}{}\thehours:%
	\ifthenelse{\value{minutes}<10}{0}{}\theminutes}

\numberwithin{equation}{section}
\numberwithin{figure}{section}
\numberwithin{table}{section}

\newtheorem{thm}{Theorem}[section]
\newtheorem{cor}[thm]{Corollary}
\newtheorem{lemma}[thm]{Lemma}
\newtheorem{prop}[thm]{Proposition}

\newtheorem{J-com}{JG-comment}[section]
\theoremstyle{definition}
\newtheorem{example}{Example}[section]

\newtheorem{defn}{Definition}[section]
\newtheorem{problem}[thm]{Problem}

\newtheorem{rem}[thm]{Remark}

\newcommand{\cG}{{\mathcal G}}
\newcommand{\Tau}{{\mathcal T}}

\def\dsum{\displaystyle\sum}

\keywords{planar ribbon graph,  partial-dual, partial-dual genus polynomials,  asymptotic normality}

\begin{document}
\title{ Partial-duals for planar ribbon graphs}

\author{Qiyao Chen}
\address{College of Mathematics, Hunan University, 410082 Changsha, China}
\email{chen1812020@163.com}
\author{Yichao Chen}
\address{College of Mathematics, Hunan University, 410082 Changsha, China}
\email{ycchen@hnu.edu.cn}

\maketitle
\begin{abstract}
\textwidth=114.3mm
{In 2009, Chmutov introduced the partial-duality for a ribbon graph $G$. Recently, Gross, Mansour and Tucker enumerated all possible partial-duals of $G$ by genus and introduced the partial-dual genus polynomial of a ribbon graph $G.$ This paper mainly enumerates partial-duals for planar ribbon graphs. First, we obtain a formula for the maximum partial-dual genus for any planar ribbon graph and give a negative answer to the interpolating conjecture of Gross, Mansour and Tucker. Then we show that  there is a recurrence relation between the partial-dual genus polynomials of planar ribbon graphs $G-e$ and $G$. Furthermore, two related results are also given.  These recurrence relations give new approaches to calculate the partial-genus dual polynomials for some planar ribbon graphs. In addition, we  prove the asymptotic normality  for some partial-dual genus distributions.}
\end{abstract}

\section{Introduction}

Ribbon graphs are often used to represent cellularly  embedded graphs. Following \cite{Chm09}, a \textit{ribbon graph} can be seen as the neighborhood of a graph embedded in the surface. It consists of two sets of  closed disks, called  \textit{vertex-disks} and \textit{edge-ribbons} (i.e. ribbons). We use $G$ to denote either  a ribbon graph or an embedded graph. Let $v(G)$, $e(G)$, $f(G)$ and $c(G)$  denote the numbers of \textit{vertices, edges, faces} and \textit{components} of $G,$ respectively. Let $\gamma(G)$ denote the \textit{orientable genus} of an oriented ribbon graph $G.$ By the Euler formula, $\gamma(G)=c(G)-\frac{1}{2}(v(G)-e(G)+f(G)).$

  We emphasize that the notation $A$ is used both for a subset of edges of a ribbon graph and as the spanning subribbon graph the edge-set of which is  $A$. The \textit{partial-dual} $G^{A},$ introduced by Chmutov in \cite{Chm09}, can be seen as geometric
duality over a partial edge set $A$ of the ribbon graph $G$.
  \begin{defn}
\cite{Chm09} Let $G$ be an  embedded graph and $A \subseteq E(G)$. Arbitrarily
orient and label each of the edges of $G$. (The orientation need not extend to an
orientation of the ribbon graph). The boundary components of the spanning ribbon
subgraph ($V(G)$, $A$) of $G$ meet the edges of $G$ in disjoint arcs (where the spanning
ribbon subgraph is naturally embedded in $G$). On each of these arcs, place an arrow
which points in the direction of the orientation of the edge boundary and is labelled by
the edge it meets. The resulting marked boundary components of the spanning ribbon
subgraph ($V(G)$, $A$) define an arrow presentation. The ribbon graph corresponding to
this arrow presentation is the partial dual $G^{A}$ of $G$. This process is explained
locally at a pair of arrows in Figure \ref{def}.

\begin{figure}[h]
  \begin{minipage}[t]{0.6\textwidth}
  \centering
  \includegraphics[width=0.8\textwidth]{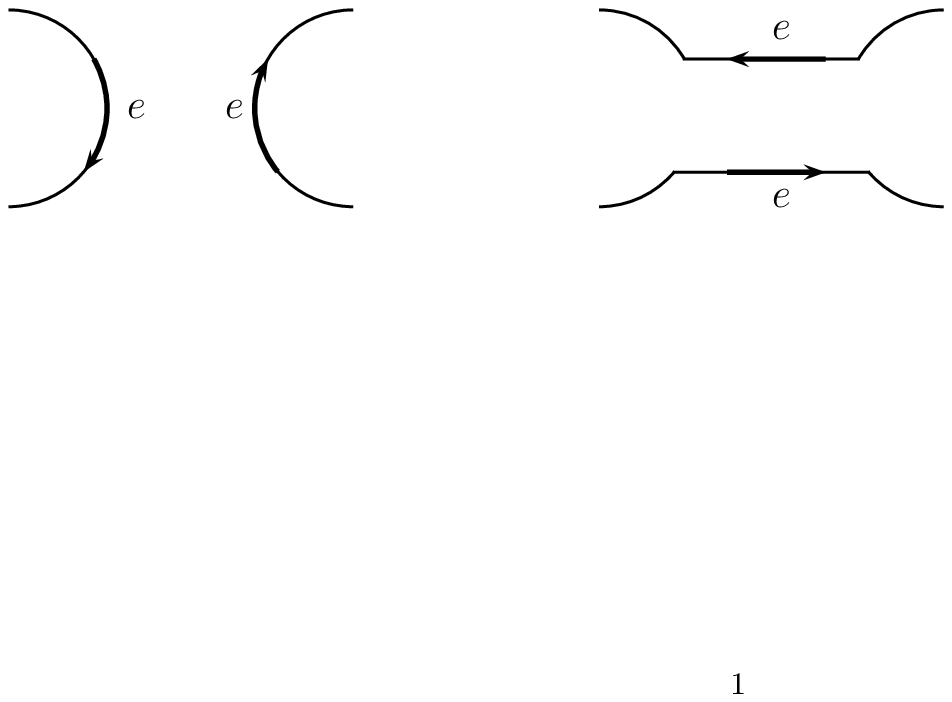}

  \end{minipage}
  \caption{ $e\in G$ (left) and $e\in G^{\{e\}}$ (right)}
  \label{def}
  \end{figure}

Partial- duality provides a way to study the relationship between knot theory, knot invariants and  ribbon graph polynomials \cite{BR02, CP07,  CV08, DFKL08, Mof18, Mof10}. We refer the readers to \cite{EM13} for more background about the partial-duality.

    \end{defn}


 \begin{prop}\cite{Chm09}
  Suppose $G$ is a ribbon graph, and let $A\subseteq E(G)$. Then the following assertions hold:
  \begin{enumerate}
\item   $v(G^{A})=f(A)$ and $e(G^{A})=e(G)$.
\item  $G^{A}$ is connected if and only if G is connected.
\item  $G^{A}$ is orientable if and only if G is orientable.
\item  $G^{\emptyset}=G$ and  $G^{E(G)}=G^{*}$.

  \end{enumerate}

  \end{prop}


Recently, Gross, Mansour and Tucker \cite{GMT18} introduced \textit{the partial-dual (orientable) genus polynomial} for any ribbon graph $G$ (\textit{pdG-polynomial}, for short), i.e., the pdG-polynomial is given by the calculation of the associated generating function of the partial-dual genus of all subgraphs of  $G.$ I.e., the pdG-polynomial of $G$ is the generating function $$~^{\partial}{\Gamma}_{G}(z)=\sum\limits_{A\subseteq E(G)}z^{\gamma[G^{A}]}.$$ They  also provided a way for expressing the \textit{orientable genus $\gamma(G^{A})$} of $G^A$. We note that $c(A)$ counts isolated vertices
of $A$ when $A$ is regarded as a spanning subgraph.

\begin{thm}\cite{GMT18}\label{Gro:main}
Let $G$ be an orientable ribbon graph and $A\subseteq E(G)$, then
\begin{eqnarray}\label{G:main}
\gamma(G^{A})&=&\gamma(A)+\gamma(A^{c})+c(G)+v(G)-c(A)-c(A^{c}).
\end{eqnarray}

\end{thm}

As noted by \cite{GMT18},   the formula above is a  variant of a result of \cite{Mof11}.

{By the \emph{partial-dual genus distribution} of a ribbon graph $G$ we mean the sequence}
$\gamma_0(G),\gamma_1(G),$ $\gamma_2(G),\cdots ,$
where $\gamma_i(G)$ is the number of partial-duals of $G$ with genus $i$ for $i\geq0$. I.e., $\sum\limits_{i\geq0}\gamma_i(G)=2^{e(G)}.$ 
The partial-dual orientable genus polynomial of $G$ is also given by  $$~^{\partial}{\Gamma}_{G}(z)=\sum\limits_{i\geq0}\gamma_i(G)z^{i}.$$

 For any ribbon graph $G,$  let  $X_{G}$ be  a random variable with distribution
\begin{eqnarray}
\label{10-1}
p_i=\mP(X_{G}=i)=\frac{\gamma_i(G) }{2^{e(G)}}, \quad i=0,1,\cdots.
\end{eqnarray}

\noindent The \textit{probability partial-dual genus polynomial} of $G$ is defined as $$P_{X_{G}}(z)=\sum\limits_{i\geq0}p_iz^{i}=\frac{~^{\partial}{\Gamma}_{G}(z)}{2^{e(G)}}.$$

Suppose that ${\cG}=\{{G}_n\}_{n=1}^\infty$ is a sequence of ribbon graphs. For $n\geq 1$, let   $e_n$ ($e_n=\mathbb{E}(X_{G_{n}})$) and  $ v_n$ be the \textit{mean} and \textit{variance} of the partial-dual genus distribution of ${G}_n$, respectively. We say  the partial-dual genus distribution of ${G}_n$ is \textit{asymptotically normal distribution}   when $n$ tends to infinity if for any $x\in \mR,$ we have
  \begin{eqnarray*}
    \lim_{n\rightarrow \infty } \mP(\frac{X_{G_n}-e_n}{ \sqrt{v_n}} \leq x)=\int_{-\infty}^x \frac{1}{\sqrt{2\pi}}e^{-\frac{1}{2}u^2}\dif u.
  \end{eqnarray*}


The article is organized as follows. In Section 2, we first obtain a Xuong-like formula \cite{Xuo79} for the maximum partial-dual genus of any planar ribbon graph. Then, using the join operation of planar ribbon graphs, we find some counter-examples to conjecture 5.3 in \cite{GMT18}. In Section 3, we obtain a recurrence relation for  the pdG-polynomials of planar ribbon graphs $G$ and $G-e$. In addition, another two related theorems are also proved. We use the theorems to compute the pdG-polynomials for some planar ribbon graphs. In Section 4,   we give asymptotic results for partial-dual distributions of some planar ribbon graphs.

\section{The maximum partial-dual genus of a planar ribbon graph }

In 1979, Xuong \cite{Xuo79} obtained a formula for the maximum genus of a graph. Here we will give a similar theorem for  the maximum partial-dual genus  of a planar ribbon graph.  In addition, a counterexample of an infinite family of ribbon graphs for the interpolation conjecture \cite{GMT18} is constructed by using planar ribbon graphs and the join operation  of two ribbon graphs.

\subsection{A formula for the maximum partial-dual genus}
 All  subgraphs discussed in this paper are spanning subgraphs. Thus,  the subgraph specified by an edge-set $A\subseteq E(G)$ is regarded as including every vertex of $G$, not just the vertices that are endpoints of edges of  $A$.
\textit{The maximum partial-dual (orientable) genus} $\gamma_{M} (G)$ of a ribbon graph $G$ is the maximum among the genera  of all partial duals of  $G$.

Let $G_{1}\cup G_{2}$ be the \emph{union} of two ribbon graphs $G_{1}$ and $G_{2}$  with vertex set $V(G_{1}\cup G_{2})$ and ribbon set $E(G_{1}\cup G_{2}).$ We denote by $G-e$ the ribbon graph obtained by removing the ribbon $e$ from $G.$ In the following discussion, we will  abbreviate $G\cup \{e\}$ to $G\cup e$.
Let $ \xi(G)=min\{c(T^{c})| T \ \rm{is\ a\ spanning\ tree\ of\ }$G$\}$.

\begin{lemma}\label{25}  Let $G$ be a connected  ribbon graph with  $A\subseteq E(G)$, then $c(A)+c(A^{c})\geq 1+\xi(G)$.
\end{lemma}
\begin{proof}If $A$ is not a forest, then we can choose a ribbon  $e$ of $A$ such that $e$ belongs to a cycle of $A$.  We have \begin{equation}
c(A)+c(A^{c})=
\begin{cases}
c(A-e)+c((A-e)^{c})+1,&\text{if e is a cut ribbon in $(A-e)^{c}$,}\\
c(A-e)+c((A-e)^{c}),&\text{otherwise.}
\end{cases}
\end{equation}
then we get $c(A)+c(A^{c})\geq c(A-e)+c((A-e)^{c}).$ Repeat the process above until $A$ becomes a forest $F$. We have $c(A)+c(A^{c})\geq c(F)+c(F^{c}).$    If $F$ is connected, then $F$ is a tree with $c(F^{c})\geq \xi(G)$, and the proof is completed. Otherwise we choose a ribbon  $e$ of $F^{c}$ such that $e$ connects different components of $F$.   In this case  $e$ is also a cut ribbon  in $F\cup e$ and $c(F\cup e)=c(F)-1$.
 \begin{equation*}
 c(F^{c})+c(F)=
 \begin{cases}
 c((F\cup e)^{c})+c(F\cup e)+2,&\text{if $e$ is a cut ribbon in $(F\cup e)^{c}$,}\\
 c((F\cup e)^{c})+c(F\cup e)+1,&\text{otherwise.}
 \end{cases}
\end{equation*}
 Repeat this process  until $F$  changes to tree  $T$, and we have $c(F)+c(F^{c})> c(T)+c(T^{c})\geq c(T)+\xi(G)$. The result follows.

\end{proof}

\begin{thm}\label{24}Let $G$ be a connected planar ribbon graph,  then $\gamma_{M} (G)=v(G)-\xi(G).$
\end{thm}
\begin{proof} Let $T$ be any spanning tree of $G$, and $A\subseteq E(G)$.  Then
$$\begin {aligned}
\hskip10pt \gamma (G^{A}) &\ \leq v(G)+c(G)-(c(T)+\xi(G)) \quad\text{by Lemma \ref{25} and Theorem \ref{Gro:main}} \\
 &\= v(G)-\xi(G)\quad\text{since $c(T)=c(G)=1$}\\
 &\=\gamma_{M} (G).
\end{aligned}
$$

\end{proof}

\begin{prop}\label{29}For $k\geq 3,$  let $e_1,e_2,\ldots,e_{k-1}$ be multiple ribbons of $e$ in a planar ribbon graph $G_k$, and let $G_{2}$ be the ribbon graph obtained by deleting $k-2$ multiple ribbons from $\{e_1,e_2,\ldots,e_{k-1}\}$ of $G_k$,  then $\gamma_{M} (G_{k})=\gamma_{M} (G_{2}).$
\end{prop}
\begin{proof}
 For any spanning tree  $T_{2}$ of $G_k$, there exists a spanning tree $T_{1}$ of $G_2$ such that $c(T_{1}^{c})= c(T_{2}^{c})$,
it follows that
$$\begin {aligned}
\hskip10pt \gamma (G^{T_{2}}_{k}) &\= 1+v(G_k)-c(T_{2})-c(T_{2}^{c}) \quad\text{by Theorem \ref{Gro:main} } \\
 &\= 1+v(G_2)-c(T_{1})-c(T_{1}^{c})\quad\text{by $v(G_2)=v(G_k)$ and $c(T_{1})=c(T_{2})$  }\\
 &\= \gamma (G^{T_{1}}_{2}).
\end{aligned}
$$
Moreover, $\xi(G_{2})=\xi(G_{k})  $, i.e.,  $\gamma_{M} (G_{k})=\gamma_{M} (G_{2}).$
\end{proof}

\begin{example}Figure \ref{fig:O_{4}} illustrates the ribbon graphs $H_{4,4},$ and $H_{4,2}.$  We observe that $H_{4,2}$ can be obtained by deleting $4$ pairs of multiple ribbons in $H_{4,4},$ by Proposition \ref{29}, $\gamma_{M} (H_{4,4}
)=\gamma_{M} (H_{4,2}
)=5.$
\end{example}

\begin{figure}[h]
  \begin{minipage}[t]{0.35\textwidth}
  \centering
  \includegraphics[width=0.6\textwidth]{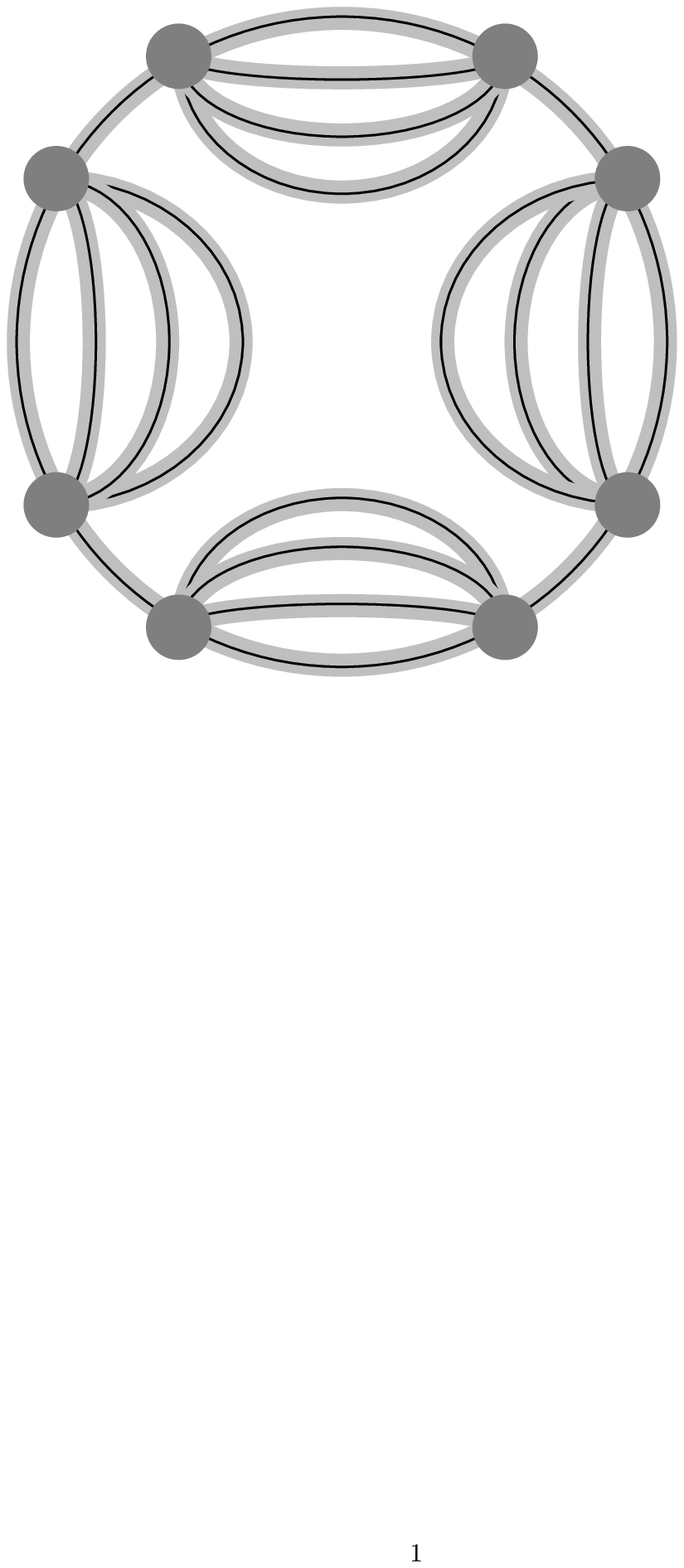}

  \end{minipage}
  \begin{minipage}[t]{0.4\textwidth}
  \centering
  \includegraphics[width=0.5\textwidth]{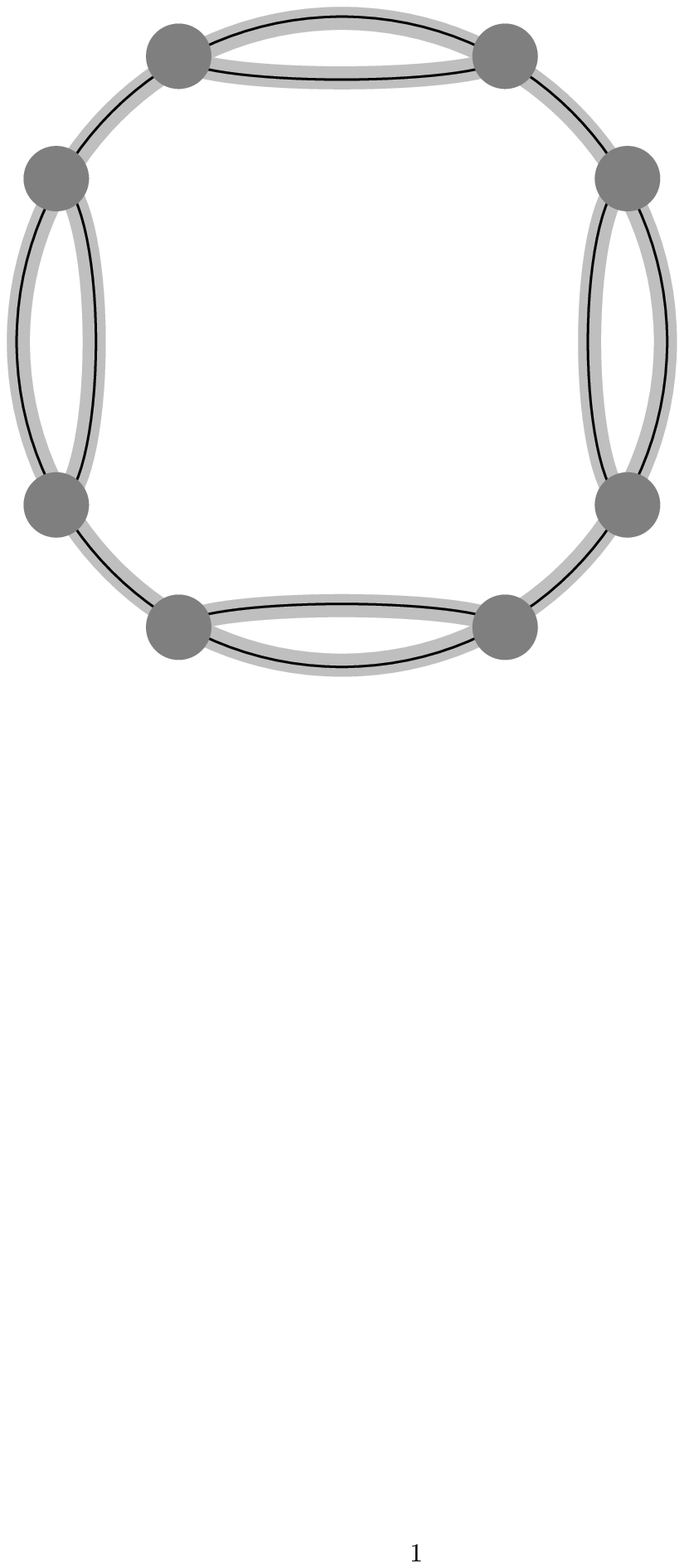}

  \end{minipage}
  \caption{ $H_{4,4}$ (left) and $H_{4,2}$ (right)}
  \label{fig:O_{4}}
  \end{figure}

We denote by $a_{\gamma_{M}}(G)$ the coefficient of $z^{\gamma_{M} (G
)}$ in the pdG-polymomial of $G.$ Let $ \Tau(G)=\{T| c(T^{c})=\xi(G),T \ \rm{is\ a\ spanning \ tree\ of\ G}\}$,   $ \Tau(G)^{c}=\{A| c(A)=\xi(G), A^{c}\ \rm{is\ a\ spanning\ tree\ of\ G}  \}$, and $ \eta(G)=\{T| T \ \rm{ and \ T^{c}\  are\ spanning\ trees}\ \\ \rm{of \ G}\}$.  The number of the set of $T$  such that  $T\in \Tau(G)$,  $T\in \Tau(G)^{c}$,  and  $T\in \eta(G)$  are denoted by $\mu(G)$, $\mu(G)^{c}$ and $|\eta(G)|$, respectively.

\begin{prop}  Let $T$ be a spanning tree of connected planar ribbon graph $G$, and let
$A\subseteq E(G)$ such that $A$, $A^{c} \neq T$. Then
\begin{enumerate}
    \item $a_{\gamma_{M}}(G)=2\mu(G),$ \quad if $ \eta(G)=\emptyset $ ~ and~ $\{A|c(A)+c(A^{c})=1+\xi(G)\}=\Tau(G)\cup\Tau(G)^{c},$
    \item $a_{\gamma_{M}}(G)> 2\mu(G),$  \quad  if $\Tau(G)\cup\Tau(G)^{c}\varsubsetneq  \{ A|c(A)+c(A^{c})=1+\xi(G)\},$
\item $a_{\gamma_{M}}(G)=| \eta(G)|,$ \quad if $ \eta(G)\neq\emptyset.$
     \end{enumerate}
 \end{prop}
\begin{proof}For  (1), because $\eta(G)=\emptyset,$ then $T^{c}$ is not a spanning tree. It is easy to see that  $\mu(G)=\mu(G)^{c}$, we have $a_{\gamma_{M}}(G)=2\mu(G).$

For (2), there is a ribbon graph $A$ that is not a spanning tree, such that $c(A)+c(A^{c})=1+\xi(G)$.\\
Now we move to prove (3). Note that if $T\in \eta(G)$ then the ribbons in $T$ and $T^{c}$ are cut edges. Adding or removing edges in $T$ will make $c(T) +c(T^{c})$ increase, and $\gamma (G^{T})=\gamma (G^{T^{c}}),$  thus  $a_{\gamma_{M}}(G)=| \eta(G)|$. The result follows.

\end{proof}

\begin{example}
Let $C_{n}$ be  the ribbon $n$-cycle  in the plane. By Example 3.3 in  \cite{GMT18},
$$~^{\partial}{\Gamma}_{C_{n}}(z)=2+(2^{n}-2)z.$$
 We observe that the loop $C_1$ satisfying
item (1) and $a_{\gamma_{M}}(C_1)=2\mu(C_1)=2.$ For $n\geq 2,$ it follows that $a_{\gamma_{M}}(C_{n})=2^{n}-2$. Since the complement of a ribbon $e$ in $E(C_2)$ is the other ribbon, it follows that $C_2$ satisfying
item (3), and $a_{\gamma_{M}}(C_2)=| \eta(C_2)|=2.$  When $n\geq 4$, let $T$ be any spanning tree of $C_{n}$, then $T\in \Tau(C_{n})$, and $\mu(C_{n})=n$, note that $a_{\gamma_{M}}(C_{n})=2^{n}-2>2n$,  it implies that $C_n$ satisfying
item (2).

\end{example}

\subsection{A counter-example to the interpolating conjecture on partial-dual Euler-genus polynomial}

  The  \textit{{join}} of two ribbon graphs $G_{1}$ and $G_{2}$, denoted by $G_{1}\vee G_{2}$, is obtained by the following two steps:
\begin{enumerate}
\item Choose an arc $p_{1}$ on the boundary of a vertex-disk $v_{1}$ of $G_{1}$ that lies between two consecutive ribbon ends, and choose another such arc $p_{2}$ on the boundary of a vertex-disk $v_{2}$ of $G_{2}$.
\item Paste vertex-disk $v_{1}$ and vertex-disk $v_{2}$ together by identifying the arcs $p_{1}$ and $p_{2}$.
 \end{enumerate}

The definition of join is given by \cite{GMT18, Mof13}.

 Given a polynomial $f(z) = \sum_{i=0}^{\infty} f_iz^i$, the \textit{spectrum }of the polynomial $f(z)$ is the set $\{i|f_i \neq 0\}$. If the spectrum of $f(z)$ is an integer interval $[m,n]$, then we say $f(z)$ is an \textit{interpolating polynomial}. We denote the \textit{partial-dual Euler-genus polynomial } of a ribbon graph G as $~^{\partial}{\Eu}_{G}(z) =\sum\limits_{A\subseteq E(G)}z^{\eu[G^{A}]},$ that enumerates partial duals by Euler-genus. We call the spectrum of $~^{\partial}{\Eu}_{G}(z)$ the \textit{partial-dual spectrum} of $G.$ In \cite{GMT18}, Gross, Mansour and Tucker proved that the pdG-polynomial is interpolating. They also conjectured that the partial-dual Euler-genus polynomial of any non-orientable graph is interpolating. Here we give a negative answer to their conjecture by using the properties of their paper.

Suppose $G$ is a planar ribbon graph with partial-dual Euler-genus polynomial $~^{\partial}{\Eu}_{G}(z) =\dsum_{i=0}^{m}z^{2k}$, where $\eu_k(G)$ is the number of partial duals of $G$ into surface $S_{2k}$ with Euler genus $2k.$ Let $B_1$ be a ribbon graph with one vertex-disk and a twisted loop. Then $B_1$ is in the projective plane and $~^{\partial}{\Eu}_{B_1}(z) = 2z.$ For $m\geq 2,$ we let $B_m = B_{m-1}\vee B_1.$ 
Then from Proposition 3.2 in \cite{GMT18}, $~^{\partial}{\Eu}_{B_m}(z) = (2z)^m,$ for $m\geq 1.$ Using Proposition 3.2 in \cite{GMT18} again, we have that the partial-dual Euler-genus polynomial for the join of a plane graph $G$ and $B_n$ equals
$~^{\partial}{\Eu}_{G\vee B_n}(z) = \dsum_{k=0}^{n}2^m\eu_k(G)z^{2k+m}$.  Thus, there exists a non-orientable ribbon graph with the partial-dual spectrum ${n,n + 2,\ldots,n + 2m}$ for  positive integers $m,n.$

\begin{rem} Recall that other counter-examples to the interpolation conjecture above were given independently by Jin and Yan \cite{YJ19}, their paper  predates ours.  They pointed out that the minimum number of edges required for the counter example is $4.$ By the construction above, we give a counter-example with $3$ edges. Let $C_2$ be the planar 2-cycle ribbon graph, it's easy to see that its partial-dual Euler-genus polynomial is $2+2z^2,$ we use it to join a twisted ribbon $B_1$, we find that $~^{\partial}{\Eu}_{C_2\vee B_1}(z)=4z+4z^3.$
\end{rem}


\section{Three theorems}
\subsection{A recursive formula for the pdG-polynomials of planar ribbon graphs }

Suppose $G$ is a connected planar ribbon graph, and let $e$ be a ribbon with $V(e)\subseteq V(G)$. We denote $G-e$ to be the ribbon graph obtained by deleting $e$ from $G,$ and  we let $ \mathscr{A}=\{A|A\cup e \rm{\ contains\ a\ cycle\ with}\ e\ ~ in~ G, A\subset E(G-e)\}.$ Here we give an example to illustrate the above definitions.
\begin{figure}[h]
  \begin{minipage}[t]{0.6\textwidth}
  \centering
  \includegraphics[width=1\textwidth]{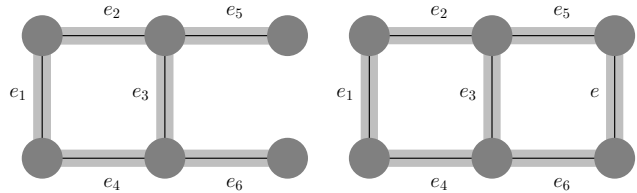}

  \end{minipage}
   \caption{$G-e$ and $G$}
  \label{fig:ladder}
  \end{figure}

\begin{example}
Suppose $G$ and $G-e$ are ribbon graphs of Figure \ref{fig:ladder}. There are two cycles containing $e$ in  $G.$ The two such cycles are $(e_{5},e_{3},e_{6},e)$ and $(e_{5},e_{2},e_{1},e_{4},e_{6},e).$   It is easy to see that
$$
\begin{aligned}
 \mathscr{A}
 &=\{\{e_{5},e_{3},e_{6}\}, \{e_{5},e_{3},e_{6},e_{1}\}, \{e_{5},e_{3},e_{6},e_{2}\}, \{e_{5},e_{3},e_{6},e_{4}\}, \{e_{5},e_{3},e_{6},e_{1},e_{2}\},\\
 &\ \ \ \ \{e_{5},e_{3},e_{6},e_{1},e_{4}\}, \{e_{5},e_{3},e_{6},e_{2},e_{4}\},\{e_{5},e_{3},e_{6},e_{1},e_{2},e_{4}\}, \{e_{5},e_{2},e_{1},e_{4},e_{6}\}\}.
 \end{aligned}
$$

\end {example}

 Now we present our main result.

 \begin{thm}\label{main}
Let $G$ be a connected planar ribbon graph, and let $e$ be one of its ribbons. If ${G-e}$ is a connected ribbon graph, then
$~^{\partial}{\Gamma}_{G}(z)=2z~^{\partial}{\Gamma}_{G-e}(z)+(2-2z)\sum\limits_{A\in\mathscr{A}} z^{\gamma[(G-e)^{A}]}$.

\end{thm}
\begin{proof}
Let $\mathscr{F}=\{ F|e\in F, F\subseteq E(G)\}$,
and let $\mathscr{F}^{c}=\{ F^{c}|  F \in \mathscr{F}\}$.  Obviously, we have
$
\gamma[G^{F^{c}}]
=\gamma[G^{F}].
 $
Thus
\begin{eqnarray}
~^{\partial}{\Gamma}_{G}(z) &=& \sum\limits_{F\in \mathscr{F}} z^{\gamma[G^{F}]}+\sum\limits_{F^{c}\in \mathscr{F}^{c}} z^{\gamma[G^{F^{c}}]}\notag\\
&=&2\sum\limits_{F\in \mathscr{F}} z^{\gamma[G^{F}]}. \label{de0}
\end{eqnarray}

Now let $F=A\cup e$, then the complement of $A$ in  $E(G-e)$ is equal to the complement of F in $E(G).$ Therefore, we have
\begin{eqnarray}
c(A^{c}) &=& c(F^{c}).    \label{de1}
\end{eqnarray}

If  $e$ is a cut ribbon in $A\cup e$,  then $A \notin\mathscr{A}$, and
 \begin{eqnarray}
c(F) &=&c(A)-1.  \label{de2}
\end{eqnarray} It follows that
$$
\begin{aligned}
\gamma[G^{F}]&=c(G)+v(G)-[c(F)+c(F^{c})]\quad \text{by (\ref{G:main})}\\
  &=c(G-e)+v(G-e)-[c(A)-1+c(A^{c})]  \quad\text{by  (\ref{de1}) and  (\ref{de2}})\\
  &=\gamma[(G-e)^{A}]+1.
  \end{aligned}
$$
and
\begin{eqnarray}
\sum\limits_{F=A\cup e,A \notin\mathscr{A}} z ^{\gamma[G^{F}]} &=&\sum\limits_{A \notin\mathscr{A}} z ^{\gamma[(G-e)^{A}]+1}\notag\\
 &=&z\sum\limits_{A \notin\mathscr{A}} z ^{\gamma[(G-e)^{A}]}.\label{de7}
 \end{eqnarray}
Otherwise, we have
 $A \in\mathscr{A}$ and
 \begin{eqnarray}
c(F) &=&c(A).  \label{de3}
\end{eqnarray}  Thus
$$
\begin {aligned}
\gamma[G^{F}]&=c(G)+v(G)-[c(F)+c(F^{c})]\quad \text{by (\ref{G:main})}\\
  &=c(G-e)+v(G-e)-[c(A)+c(A^{c})]  \quad\text{by  (\ref{de1}) and  (\ref{de3}}) \\
  &=\gamma[(G-e)^{A}].
  \end {aligned}
$$
Moreover,
\begin{eqnarray}
\sum\limits_{F=A\cup e,A\in\mathscr{A}} z ^{\gamma[G^{F}]} &=&\sum\limits_{A\in\mathscr{A}} z ^{\gamma[(G-e)^{A}]}.  \label{de5}
\end{eqnarray}
It is easy to show that
\begin{eqnarray}
~^{\partial}{\Gamma}_{G-e}(z) &=&\sum\limits_{A\notin\mathscr{A}} z ^{\gamma[(G-e)^{A}]}+\sum\limits_{A\in\mathscr{A}} z^{\gamma[(G-e)^{A}]}.  \label{de4}
\end{eqnarray}
From the discussions above, we have

\begin{eqnarray}\label{mmmain}
~^{\partial}{\Gamma}_{G}(z) &=& 2\sum\limits_{F\in \mathscr{F}} z^{\gamma[G^{F}]}\quad\text{by  (\ref{de0})}\notag \\
&=& 2(\sum\limits_{F=A\cup e,A \notin\mathscr{A}} z ^{\gamma[G^{F}]} +\sum\limits_{F=A\cup e,A\in\mathscr{A}} z ^{\gamma[G^{F}]})\notag \\
&=& 2(z\sum\limits_{A \notin\mathscr{A}} z ^{\gamma[(G-e)^{A}]} +\sum\limits_{A\in\mathscr{A}} z^{\gamma[(G-e)^{A}]}) \quad\text{by  (\ref{de7})~ and~(\ref{de5}) }\notag \\
&=&2z~^{\partial}{\Gamma}_{G-e}(z)+(2-2z)\sum\limits_{A\in\mathscr{A}} z^{\gamma[(G-e)^{A}]}\quad\text{by  (\ref{de4}) }
\end{eqnarray}
The result follows.

\end{proof}

Similarly, we have the following theorem.
\begin{thm}\label{main1}
Let $G$ be a connected planar ribbon graph and $e$ be one of its ribbons. If ${G-e}$ is a connected ribbon graph, then
$~^{\partial}{\Gamma}_{G}(z)=2~^{\partial}{\Gamma}_{G-e}(z)+(2z-2)\sum\limits_{A\in\mathscr{A}} z^{\gamma[(G-e)^{A}]},$ where $\mathscr{A}=\{A| e \rm{\ is\ a\ cut\ ribbon\ in}\ A\cup e, A\subset E(G-e)\}.$

\end{thm}

\begin{cor}\label{1/2}
Let $G$ be a connected planar ribbon graph and  $e$ be one of its ribbons, then
 $\sum\limits_{e\in F,F\subseteq E(G)} z^{\gamma[G^{F}]}
=\frac{1}{2} ~^{\partial}{\Gamma}_{G}(z)$.
\end{cor}
\begin{proof}
Let $\mathscr{F}=\{ F|e\in F, F\subseteq E(G)\}$,  thus the corollary follows from  (\ref{de0}).

\end{proof}

\subsection{Planar ribbon graphs with multiple ribbons} Here we will give a recursive formula for the pdG-polynomials of planar ribbon graphs with multiple edges.

\begin{thm}\label{multiple}
Let $G$ be a connected planar ribbon graph and let  $e_{1}$ be one of its ribbons. Let $e_{i}$ be the multiple ribbon of $e_{1}$, $2\leq i\leq n$, and let $G\cup \{e_{2},\cdots, e_{n} \}$ be the planar ribbon graph obtained by inserting edges $e_{i}$ parallel to  $e_{1}$.
If $G-e_{1}$ is a connected  ribbon graph, then
\begin{enumerate}

\item Let $\mathscr{A}_{1}=\{A| A\cup e_{2} \rm{\ contains\ a\ cycle\ with\ e_{2}}  \rm{\ in}\  G\cup e_{2}, e_{1}\notin A, A\subset E(G)\}$. If $e_{1}$ is a cut ribbon in  $A^{c},$ for any $A\in \mathscr{A}_{1},$ then
\begin{gather}
~^{\partial}{\Gamma}_{G\cup e_{2}}(z)= (2z+1)~^{\partial}{\Gamma}_{G}(z)-2z^{2}~^{\partial}{\Gamma}_{G-e_{1}}(z),\label{eq:11}
\end{gather}
\noindent otherwise  \begin{gather}
\ ~^{\partial}{\Gamma}_{G\cup e_{2}}(z)=(2z+1)~^{\partial}{\Gamma}_{G}(z)-2z^{2}~^{\partial}{\Gamma}_{G-e_{1}}(z)+2(1-z)^{2}\sum\limits_{A\in \mathscr{A}_{2}}  z^{\gamma[G^{A}]},\label{eq:12}
\end{gather}
where  $\mathscr{A}_{2}=\{A| A\cup e_{2} \rm{\ contains\ a\ cycle\ with}\ e_{2}\ \rm{\ and}\ A^{c} \rm{\ contains\ a}\ \rm{cycle\ with}$ $\ e_{1}\ \rm{\ in}\  G\cup e_{2},  A\subset E(G)   \}.$
\item For $n\geq 3,$ then
\begin{eqnarray}
\ ~^{\partial}{\Gamma}_{G\cup \{e_{2},\cdots, e_{n} \}}(z)&= &(2^{n-1}-1) ~^{\partial}{\Gamma}_{G\cup e_{2}}(z)-(2^{n-1}-2) ~^{\partial}{\Gamma}_{G}(z).\notag
\end{eqnarray}
\end{enumerate}

\end{thm}
\begin{proof}
Here we define  $\mathscr{A}=\{A| A\cup e_{2} \rm{\ contains\ a\ cycle\ with}\ e_{2}~ in~  G\cup e_{2},   A\subset E(G)\}$, and $\mathscr{\overline{A}}=\{\overline{A}| \overline{ A}\cup e_1 \rm{\ contains\ a\ cycle\ with}\ e_{1}\  \rm{\ in}\  G,   \overline{A}\subset E(G-e_{1})\}$. We observe that two multiple edges $e_{2}$ and $e_{1}$ is a $2$-cycle in  $G\cup e_{2}$,
then
we let $\mathscr{A}_{3}=\{A|  e_{1}\in A,  A\subset E(G)\}$. Thus, applying Corollary \ref{1/2}, we have \begin{eqnarray}
\sum\limits_{A\in \mathscr{A}_{3}} z^{\gamma[G^{A}]}&=& \frac{1}{2} ~^{\partial}{\Gamma}_{G}(z). \label{m1}
\end{eqnarray}
For any  planar multiple ribbon graph $G\cup e_{2}$, we partition the calculation of the item $(1)$ into  two cases.

$\mathbf{Case~ 1}:$  Assume $e_{1}$ is a cut ribbon in $A^{c}$, for any $A\in \mathscr{A}_{1}$. Now, the ribbon subsets of $G\cup e_{2}$ forming a circle with $e_{2}$ are divided into two cases: $A\in \mathscr{A}_{1}$ and $A\in \mathscr{A}_{3}$, i.e.,  $\mathscr{A}=\mathscr{A}_{1}\cup \mathscr{A}_{3}$. For any $A\in \mathscr{A}_{1}$, there exists $\overline{A}\in \mathscr{\overline{A}}$, such that $A=\overline{A}$. Nevertheless, $A^{c}=\overline{A}^{c}\cup e_{1},$  thus,
\begin{eqnarray}
c(A^{c})&=& c(\overline{A}^{c})-1, \label{m41}
\end{eqnarray}
 \begin{eqnarray}
\gamma[G^{A}]&=&c(G)+v(G)-c(A)-c(A^{c})\quad\text{by Theorem~\ref{Gro:main}}\notag\\
&=&c(G-e_{1})+v(G-e_{1})-c(\overline{A})-c(\overline{A}^{c})+1 \quad\text{by (\ref{m41})}\notag\\
&=&\gamma[(G-e_{1})^{\overline{A}}]+1. \label{m2}
\end{eqnarray}

 Moreover,  we have
 \begin{eqnarray}
 \sum\limits_{A\in \mathscr{A}} z^{\gamma[G^{A}]}&=&\sum\limits_{A\in \mathscr{A}_{1}} z^{\gamma[G^{A}]}+\sum\limits_{A\in \mathscr{A}_{3}} z^{\gamma[G^{A}]}\notag\\
&=&z\sum\limits_{\overline{A}\in \mathscr{\overline{A}}} z^{\gamma[(G-e_{1})^{\overline{A}}]}+\frac{1}{2}~^{\partial}{\Gamma}_{G}(z) \quad\text{by (\ref{m1}) and (\ref{m2})}. \label{m3}
\end{eqnarray}

$\mathbf{Case~ 2}:$ There exists  $A\in \mathscr{A}_{1},$ such that $e_{1}$ is  not a cut ribbon in  $A^{c}$. That is,  $e_{1}$ in $A^{c}$ is either a cut ribbon or a ribbon on the circle, thus we let $\mathscr{A}_{2}=\{A| A\cup e_{2} \rm{\ contains\ a\ cycle\ with\ e_{2}\ \ and\ A^{c}\  contains\ a\ cycle\ with \ e_{1}\ in}\ G\cup e_{2},  A\subset E(G)   \}$. For any $A\in \mathscr{A}_{1}$ or $A\in \mathscr{A}_{2}$,
 there exists $\overline{A}$ such that $A=\overline{A}$, and $A^{c}=\overline{A}^{c}\cup e_{1}$. In other words,   when $A\in \mathscr{A}_{2}$,   we have
\begin{eqnarray}
c(A^{c})&=& c(\overline{A}^{c}), \label{m5} \quad\text{by the definition of $\mathscr{A}_{2}$}
\end{eqnarray}
and
\begin{eqnarray}
\gamma[G^{A}]&=&c(G)+v(G)-c(A)-c(A^{c})\quad\text{by Theorem~\ref{Gro:main}}\notag\\
&=&c(G-e_{1})+v(G-e_{1})-c(\overline{A})-c(\overline{A}^{c}) \quad\text{by (\ref{m5})}\notag\\
&=&\gamma[(G-e_{1})^{\overline{A}}]. \label{m22}
\end{eqnarray}

When $A\in \mathscr{A}_{1}$, this statement can be proved in the same way as shown in Case $1$.
 Consequently, by (\ref{m2}) and (\ref{m22}), we infer that
\begin{eqnarray}\label{m6}
\sum\limits_{\overline{A}\in \mathscr{\overline{A}}} z^{\gamma[(G-e_{1})^{\overline{A}}]}&=&\frac{1}{z}\sum\limits_{A\in \mathscr{A}_{1}} z^{\gamma[G^{A}]}+\sum\limits_{A\in \mathscr{A}_{2}}z^{\gamma[G^{A}]}.
\end{eqnarray}
 According to the previous analysis,  we divide the ribbon subsets of $G\cup e_{2}$ forming a circle with $e_{2}$  into three cases: $A\in \mathscr{A}_{1}$,  $A\in \mathscr{A}_{2}$ and $A\in \mathscr{A}_{3}$, i.e., $\mathscr{A}=\mathscr{A}_{1}\cup  \mathscr{A}_{2} \cup \mathscr{A}_{3}$.

Furthermore,
\begin{eqnarray}\label{E:1}
\sum\limits_{A\in \mathscr{A}} z^{\gamma[G^{A}]}&=&\sum\limits_{A\in \mathscr{A}_{1}} z^{\gamma[G^{A}]}+\sum\limits_{A\in \mathscr{A}_{2}} z^{\gamma[G^{A}]}+\sum\limits_{A\in \mathscr{A}_{3}} z^{\gamma[G^{A}]}\notag\\
&=&z(\sum\limits_{\overline{A}\in \mathscr{\overline{A}}} z^{\gamma[(G-e_{1})^{\overline{A}}]}-\sum\limits_{A\in \mathscr{A}_{2}} z^{\gamma[G^{A}]})+\sum\limits_{A\in \mathscr{A}_{2}} z^{\gamma[G^{A}]} \notag\\
 &&\+\frac{1}{2} ~^{\partial}{\Gamma}_{G}(z)\quad\text{by (\ref{m6})}\notag\\
&=&z\sum\limits_{\overline{A}\in \mathscr{\overline{A}}} z^{\gamma[(G-e_{1})^{\overline{A}}]}+(1-z)\sum\limits_{A\in \mathscr{A}_{2}} z^{\gamma[G^{A}]}+\frac{1}{2} ~^{\partial}{\Gamma}_{G}(z)
\end{eqnarray}

Theorem \ref{main} implies that
\begin{eqnarray}
~^{\partial}{\Gamma}_{G\cup e_{2}}(z)&=&2z~^{\partial}{\Gamma}_{G}(z)+(2-2z)\sum\limits_{A\in \mathscr{A}} z^{\gamma[G^{A}]},\label{eq:3}\\
~^{\partial}{\Gamma}_{G}(z)&=&2z~^{\partial}{\Gamma}_{G-e_{1}}(z)+(2-2z)\sum\limits_{\overline{A}\in \mathscr{\overline{A}}} z^{\gamma[(G-e_{1})^{\overline{A}}]}.\label{eq:4}
\end{eqnarray}
Combining  (\ref{m3}), (\ref{eq:3}) and
(\ref{eq:4}), we can get (\ref{eq:11}).
By (\ref{E:1})-(\ref{eq:4}),
 we get (\ref{eq:12}).

Now we give a proof for item (2). For simplicity, we may take $G_{n}=G-e_{1}\cup \{e_{1},e_{2},\cdots, e_{n} \}$, note that $G_{1}=G$ and $G_{2}=G\cup e_{2}$.  Here we let $ \mathscr{C}_{i}=\{A_{i}|  e_{i+1} \rm{\ is\ a\ cut\ ribbon\ in}\ A_{i}\cup e_{i+1},\  A_{i}\subset E(G_{i})\},$ where $i= 1,\cdots, n-1.$   For each $A_{n-1}\in\mathscr{C}_{n-1}$, we can choose  $A_{1}\in\mathscr{C}_{1}$, such that $A_{1}=A_{n-1}$.
However, the complement of $A_{1}$ in $E(G_{1})$ is obtained from the complement of
 $A_{n-1}$ in $E(G_{n-1})$ by deleting the multiple edges, and deleting the multiple edges does not change the number of components,
 thus,
\begin{eqnarray}
c(A_{n-1}^{c})&=& c(A_{1}^{c}), \label{m4}
\end{eqnarray}
 \begin{eqnarray}
\gamma[(G_{n-1})^{A_{n-1}}]&=&c(G_{n-1})+v(G_{n-1})-c(A_{n-1})-c(A_{n-1}^{c})\quad\text{by Theorem~\ref{Gro:main}}\notag\\
&=&c(G_{1})+v(G_{1})-c(A_{1})-c(A_{1}^{c}) \quad\text{by (\ref{m4})}\notag\\
&=&\gamma[G_{1}^{A_{1}}].\label{m8}
\end{eqnarray}
 and
 \begin{eqnarray}\label{E:3}
\sum\limits_{A_{n-1}\in \mathscr{C}_{n-1}} z^{\gamma[(G_{n-1})^{A_{n-1}}]}
&=&\sum\limits_{A_{1}\in \mathscr{C}_{1}} z^{\gamma[G_{1}^{A_{1}}]}\quad\text{by (\ref{m8})}\notag\\
&=& \frac{~^{\partial}{\Gamma}_{G_{2}}(z)-2~^{\partial}{\Gamma}_{G_{1}}(z)}{2z-2}\quad\text{by Theorem \ref{main1}}
\end{eqnarray}
Finally, we have

\begin{eqnarray}
~^{\partial}{\Gamma}_{G_{n}}(z)
&=&2~^{\partial}{\Gamma}_{G_{n-1}}(z)+(2z-2)\sum\limits_{A_{n-1}\in \mathscr{C}_{n-1}} z^{\gamma[(G_{n-1})^{A_{n-1}}]}\quad\text{by Theorem \ref{main1}}\notag\\
&=&2~^{\partial}{\Gamma}_{G_{n-1}}(z)+(2z-2)\frac{~^{\partial}{\Gamma}_{G_{2}}(z)-2~^{\partial}{\Gamma}_{G_{1}}(z)}{2z-2} \quad\text{by  (\ref{E:3})}\notag\\
&=&2~^{\partial}{\Gamma}_{G_{n-1}}(z)+~^{\partial}{\Gamma}_{G_{2}}(z)-
2~^{\partial}{\Gamma}_{G_{1}}(z)\notag\\
&=&(2^{n-1}-1) ~^{\partial}{\Gamma}_{G\cup e_{2}}(z)-(2^{n-1}-2) ~^{\partial}{\Gamma}_{G}(z),
\end{eqnarray} where $n\geq 3.$

\end{proof}

A \textit{subdivision} of a ribbon graph $G$ is obtained by replacing  ribbion $e=uv$ by a path $uwv.$
\begin{thm}\label{wn}
\cite{GMT18} Given a ribbon graph $G$ and a ribbon $e.$ Let $K$ be a subdivision of $G,$ then\\
\begin{equation}
~^{\partial}{\Gamma}_{K}(z)=
\begin{cases}
 2~^{\partial}{\Gamma}_{G}(z),&\text{if e is a cut ribbon,}\\
 ~^{\partial}{\Gamma}_{G}(z)+(2z)~^{\partial}{\Gamma}_{G-e}(z),&\text{if e is non-separating.}
\end{cases}
\end{equation}
\end{thm}

\begin{example}\label{thm:2m2}
A \textit{suspension} of two graphs $G_1$ and $G_2$, denoted by $G_{1}+G_{2},$ is obtained by adjoining each vertex of $G_{1}$  to each vertex of $G_{2}$. Let $P_{m+2}=v_{2}u_{1}u_{2}\cdots u_{m}v_{3}$ be a path graph. A \textit{standard fan graph} $F_{(2,m,2)}$ is obtained by adding  multiple edge of $v_{1}v_{2}$ and multiple edge of $v_{1}v_{3}$ to  $P_{m+2}+\{v_1\},$ as shown in Figure \ref{fig:F_{2}}.

\begin{figure}[h]
  \centering
  \includegraphics[width=0.4\textwidth]{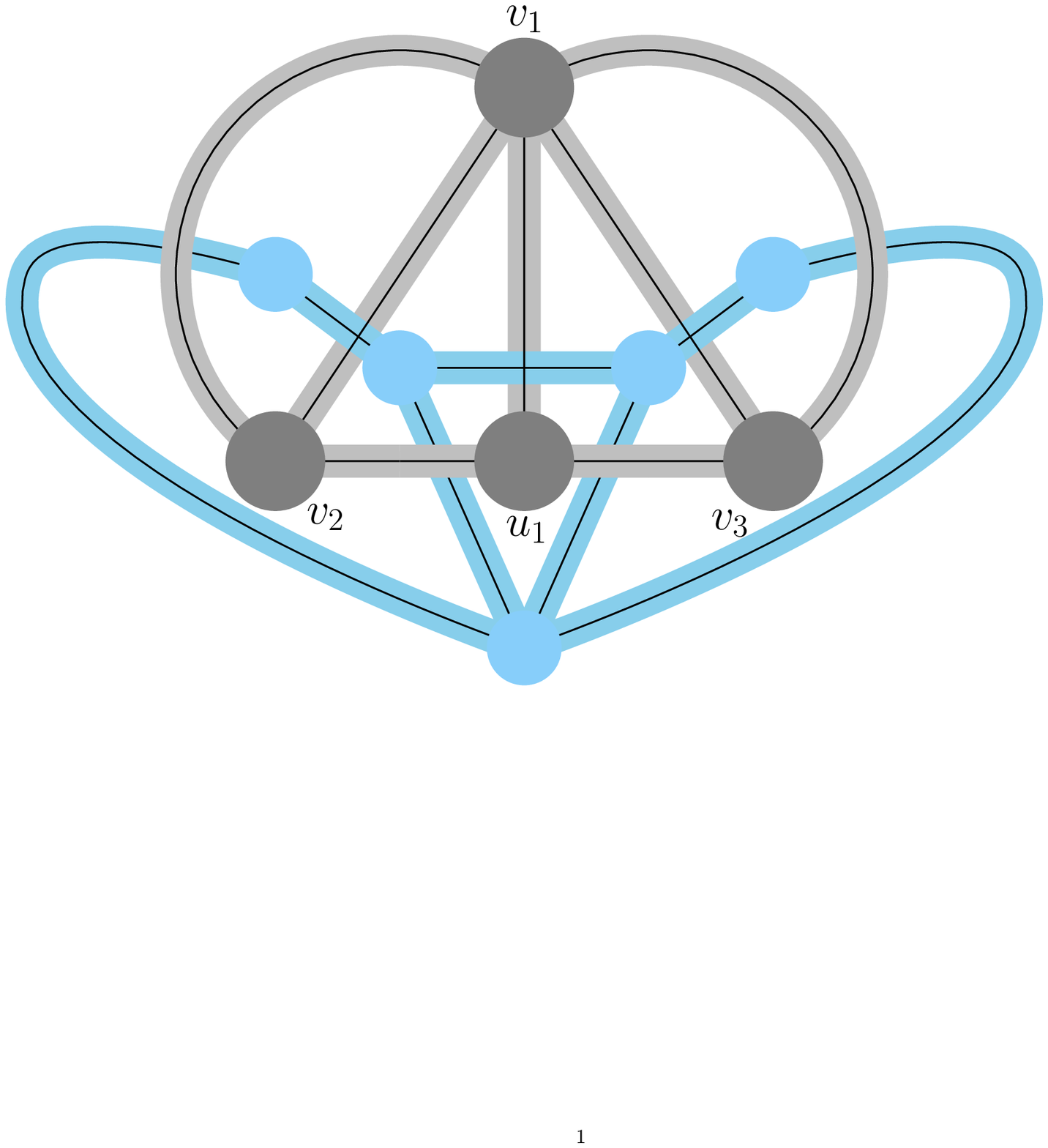}
 \caption{$F_{(2,1,2)}$ and it's dual graph $F_{(2,1,2)}^{*}$}
 \label{fig:F_{2}}
\end{figure}

  It is clear that the  ribbon graph $F_{(2,m,2)}  ^{*}$ is isomorphic to $F_{(1,m+1,1)}$. Let $Q_1$ be a ribbon  and let $Q_{m+3}=F_{(1,m+1,1)}$. It's easy to see that there are $m+2$ $3$-cycles in $Q_{m+3}$.
 \begin{figure}[h]
\centering
\includegraphics[width=1.7in]{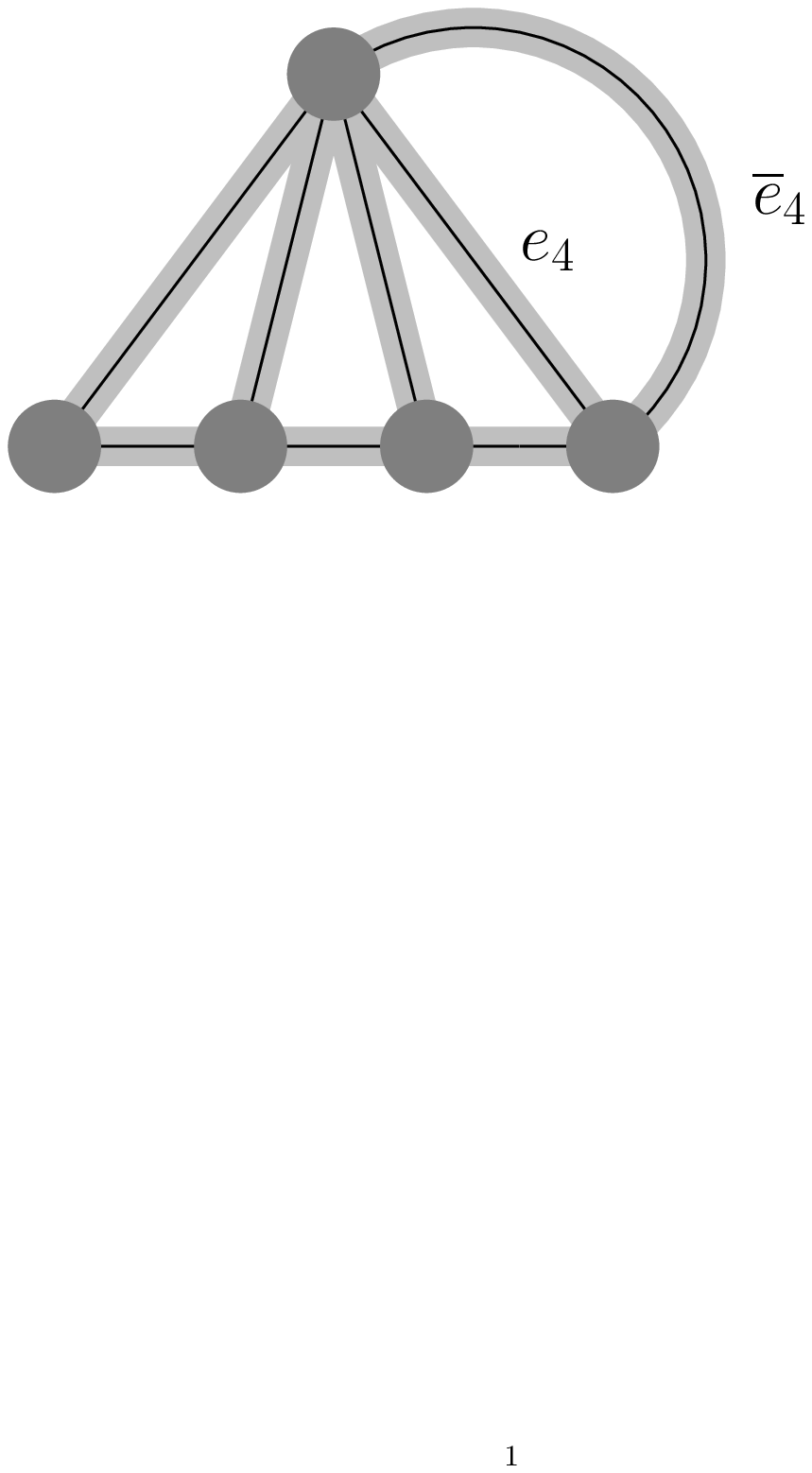}
\includegraphics[width=1.5in]{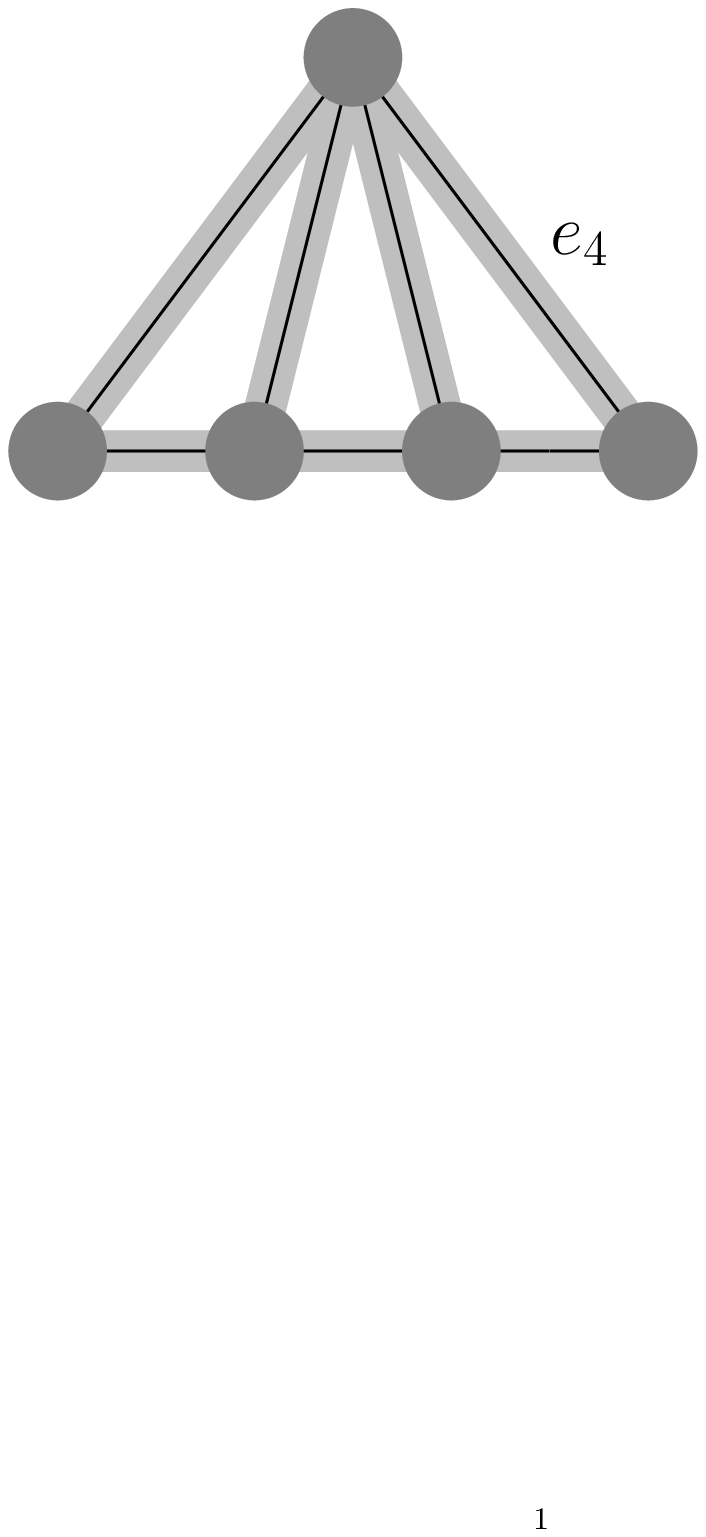}

 \caption{$\overline{Q}_{4}$ and $Q_{4}$}
  \label{fig:O_{5}}
\end{figure}
 By adding a multiple ribbon $\overline{e}_{n}$ to $e_{n}$ in $Q_{n},$ the resulting ribbon graph is $\overline{Q}_{n}$, as shown in Figure \ref{fig:O_{5}}. By  Proposition 3.2 in  \cite{GMT18},
\begin{eqnarray}\label{E:8}
~^{\partial}{\Gamma}_{\overline{Q}_{n}}(z)&\=(2z+1)~^{\partial}{\Gamma}_{Q_{n}}(z)-2z^{2}~^{\partial}{\Gamma}_{Q_{n}-e_{n}}(z)
&\text{by (\ref{eq:11}) }\notag\\
&\=(2z+1)~^{\partial}{\Gamma}_{Q_{n}}(z)-4z^{2}~^{\partial}{\Gamma}_{Q_{n-1}}(z).
\end{eqnarray}
Since $Q_{n+1}$ can be view as $\overline{Q}_{n}$ by subdividing ribbon $\overline{e}_{n}$  once,  by Theorem \ref{wn} and Equation (\ref{E:8}), we get
\begin{equation}
\begin{split}
 ~^{\partial}{\Gamma}_{Q_{n+1}}(z)&=~^{\partial}{\Gamma}_{\overline{Q}_{n}}(z)+2z~^{\partial}{\Gamma}_{Q_{n}}(z)
 \notag\\
&=(4z+1)~^{\partial}{\Gamma}_{Q_{n}}(z)-4z^{2}~^{\partial}{\Gamma}_{Q_{n-1}}(z)
\end{split}
\end{equation}
\noindent with initial conditions $~^{\partial}{\Gamma}_{Q_{1}}(z)=2,$ and $~^{\partial}{\Gamma}_{Q_{2}}(z)=6z+2.$ Solving the equation above, we get \begin{equation}\label{equ:add:Qn}\begin{split}~^{\partial}{\Gamma}_{F_{(2,m,2)}}(z)&=~^{\partial}{\Gamma}_{Q_{m+3}}(z)\\
&= \sum \limits_{k= 0}^{m+3}g(k,m+2-k)- zg(k,m+1-k),\end{split}\end{equation} where
\begin{equation}
g(n,k)=
\begin{cases}
 (-1)^{k}2^{2k+1}(^{n}_{k})(1+4z)^{n-k}z^{2k},&\text{if $0\leq k\leq n,$}\\
 0,&\text{if $k> n$ or $k< 0.$}
\end{cases}
\end{equation}

\end{example}

The graph formed by adding an additional vertex adjacent to each vertex on $C_{n}$ is called a \textit{wheel} graph $W_{n}$. By attaching a multiple edge
$\overline{e}_{2n-1}$ to $e_{2n-1}$ in $W_{n}$, the resulting graph is $\overline{W}_{n}$. A labeling of the ribbons of $W_{n}$ and $\overline{W}_{n}$ is shown in Figure \ref{fig:w_{n}}.

\begin{example}\label{w_{n}}
For $n\geq 3$, let $W_{n}$ be a planar ribbon wheel graph. An outline of the proof of Example \ref{w_{n}} is as follows.
\begin{figure}[h]

  \begin{minipage}[t]{0.55\textwidth}
  \centering
  \includegraphics[width=.99\textwidth]{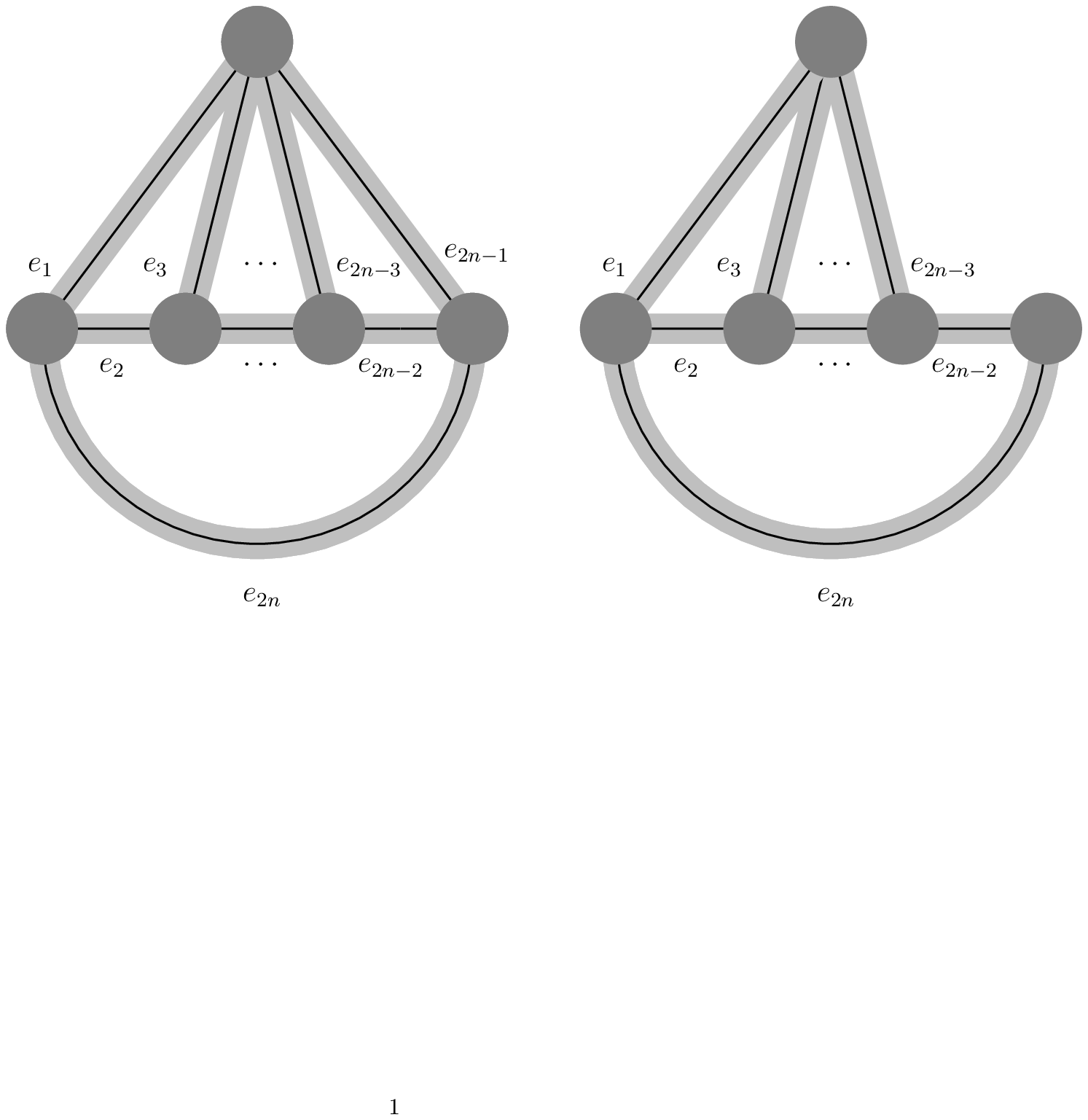}

  \end{minipage}
  \begin{minipage}[t]{0.34\textwidth}
  \centering
  \includegraphics[width=1\textwidth]{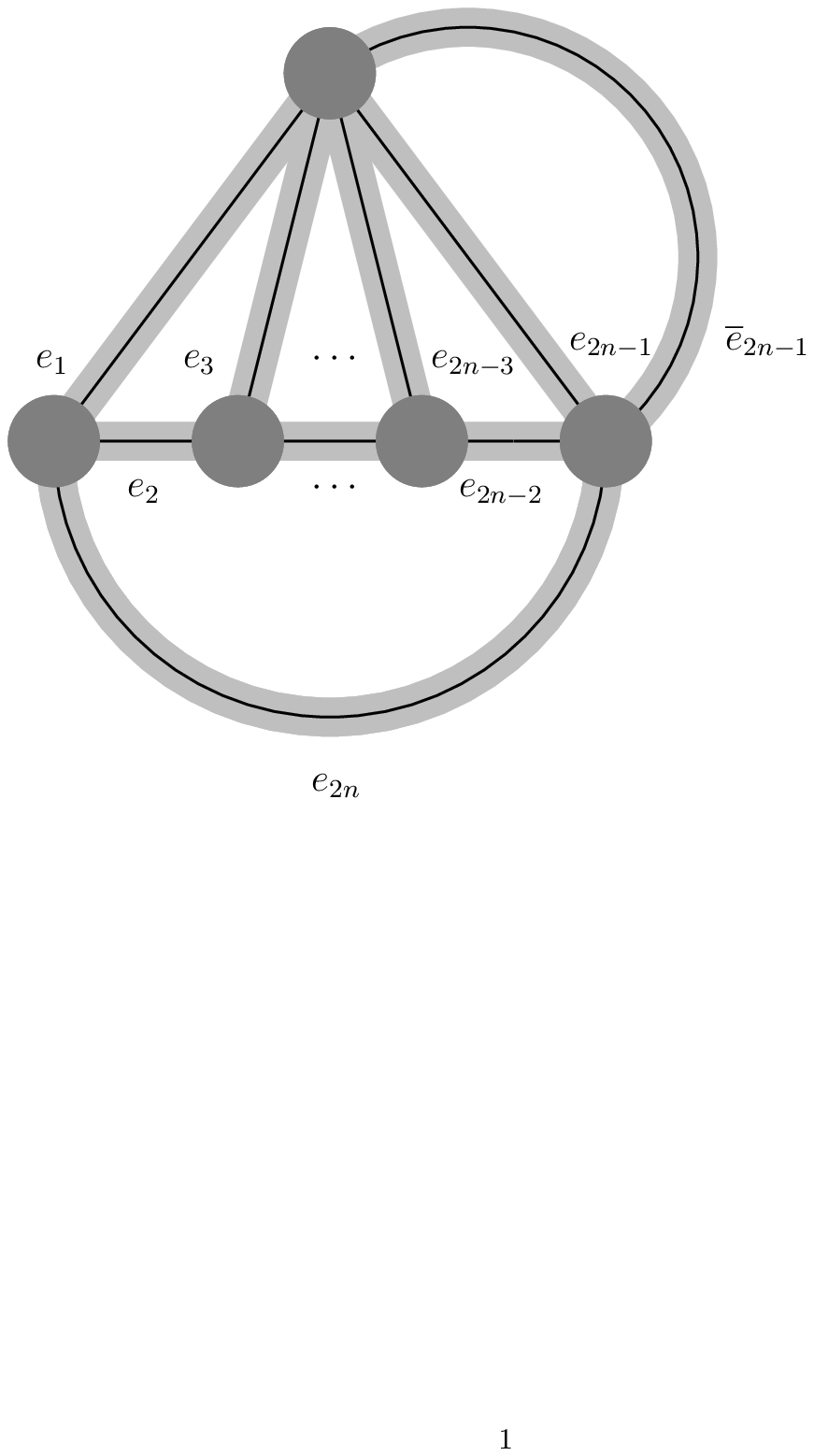}

  \end{minipage}
  \caption{$W_{n}$, $W_{n}-e_{2n-1}$ and $\overline{W}_{n}$ }
  \label{fig:w_{n}}
  \end{figure}
\begin{figure}[h]
  \begin{minipage}[t]{0.4\textwidth}
  \centering
  \includegraphics[width=0.7\textwidth]{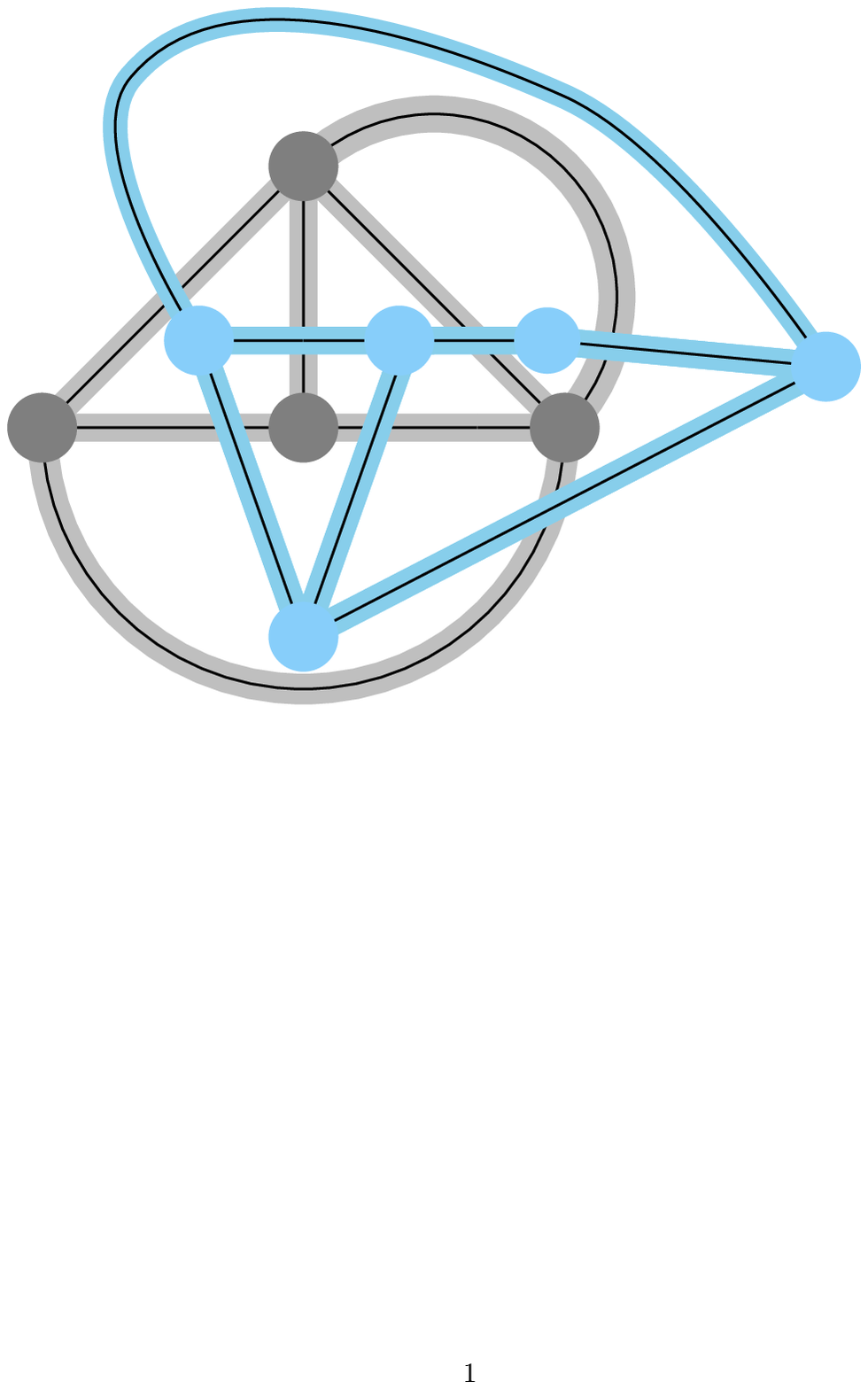}

  \end{minipage}
  \begin{minipage}[t]{0.35\textwidth}
  \centering
  \includegraphics[width=0.75\textwidth]{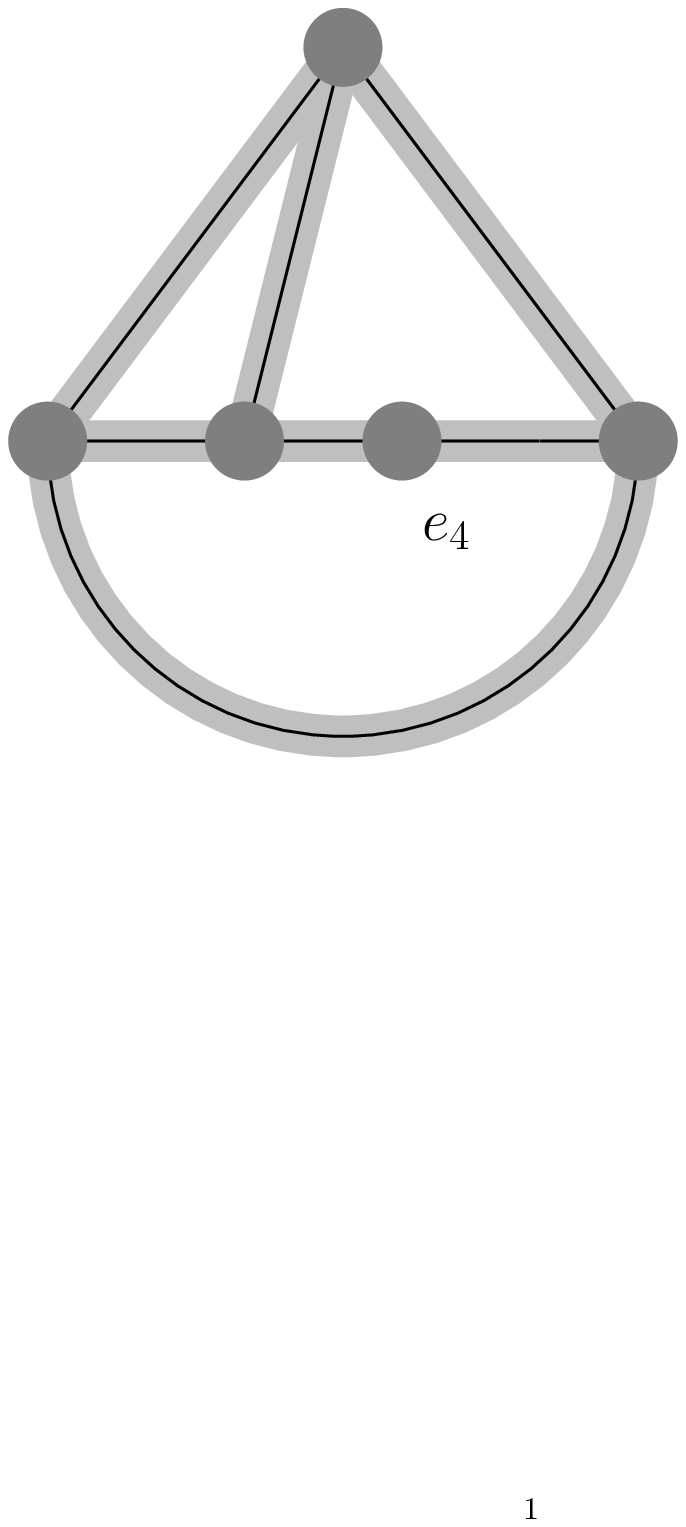}

  \end{minipage}
  \caption{$\overline{W}_{3}^{*}$}
   \label{fig:w_{5}}
  \end{figure}

 First, we give a formula between the pdG-polynomial  of $\overline{W}_{n}^{*}$ and  $W_{n}$. Since $\overline{W}_{n}^{*}$ is isomorphic to $W_{n}$ with ribbon $e_{2n-2}$ subdivided once, as shown in Figure \ref{fig:w_{5}}, and $W_{n}-e_{2n-2}$ is isomorphic to $Q_{n}$(see Example \ref{thm:2m2} for the definition), as shown in Figure \ref{fig:w_{4}}, we know
\begin{eqnarray}\label{E:18}
~^{\partial}{\Gamma}_{\overline{W}_{n}^{*}}(z)&=&~^{\partial}{\Gamma}_{W_{n}}(z)+2z~^{\partial}{\Gamma}_{W_{n}-e_{2n-2}}(z)
\quad\text{by  Theorem \ref{wn}}\notag\\
&=&~^{\partial}{\Gamma}_{W_{n}}(z)+2z~^{\partial}{\Gamma}_{Q_{n}}(z).
\end{eqnarray}

  \begin{figure}[h]
  \begin{minipage}[t]{0.7\textwidth}
  \centering
  \includegraphics[width=0.75\textwidth]{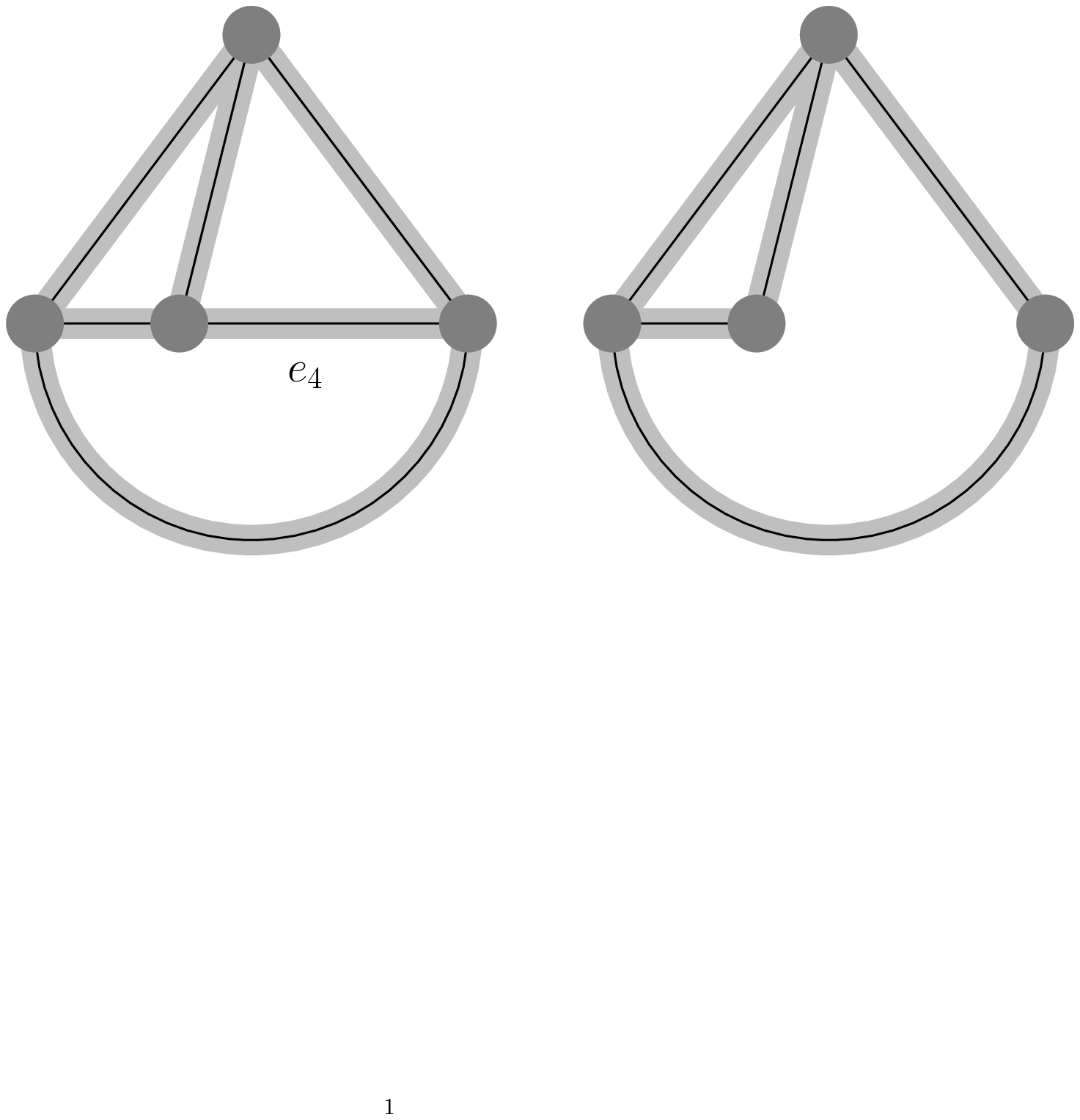}

  \end{minipage}

  \caption{$ W_{3}$ and $W_{3}-e_{4}$}
   \label{fig:w_{4}}
  \end{figure}
 Next we give a formula between the pdG-polynomial  of $\overline{W}_{n}$ and  $W_{n}$.
 Let  $\mathscr{A}=\{A| A^{c}\cup \overline{e}_{2n-1} \rm{\ contains\ a\ cycle\ with \ \overline{e}_{2n-1}\  and\ A \ contains\ a\ cycle} \rm{\ with}$ $\rm{\
 e_{2n-1}}  \rm{\ in}$ $ \rm{\ \overline{W}_{n},}\ A\subset E(W_{n})\},$ for $n\geq 3.$ The ribbon graph $W_{n}$ can be obtained from $\overline{W}_{n}$ by deleting the multiple ribbon $\overline{e}_{2n-1}$,   equation (\ref{eq:12})  gives
\begin{equation}\label{eq:16}
~^{\partial}{\Gamma}_{\overline{W}_{n}}(z)=(2z+1)~^{\partial}{\Gamma}_{W_{n}}(z)-2z^{2}
~^{\partial}{\Gamma}_{W_{n}-e_{2n-1}}(z)+2(1-z)^{2}\sum\limits_{A\in \mathscr{A}} z^{\gamma[W_{n}^{A}]}.
\end{equation}

Our problem reduces to calculate $\sum\limits_{A\in \mathscr{A}} z^{\gamma[W_{n}^{A}]}$ and $~^{\partial}{\Gamma}_{W_{n}-e_{2n-1}}(z)$. For any $A^{c}\in \mathscr{A}$, we partition the calculation of
 $\sum\limits_{A\in \mathscr{A}} z^{\gamma[W_{n}^{A}]}$ into two cases: $e_{2n}\in A^{c}$ or $e_{2n}\notin A^{c}$. First, suppose that $e_{2n}\in A^{c}$ and $e_{2n-1}\in A$.
Put $\mathscr{C}_{i}=\{A_{i}|e_{1}\notin A_{i} \rm{\ and}\ e_{2i-1}\in A_{i}, A_{i}\subseteq E(Q_{i})\}$,  and
 let $f_{i}(z)=\sum\limits_{A_{i}\in \mathscr{C}_{i}} z^{\gamma[Q_{i}^{A_{i}}]}$, where $i\geq 2$.
It is easy to see that $f_{2}(z)=2z$. The ribbon graph ${A_{n}}-e_{2n-1}$ can be viewed as joining a ribbon $e_{2n-2}$ to a subgraph of $Q_{n-1}$ without the edge $e_1$(the label of $Q_{n}$ is the same as the label of $W_{n}$ in Figure \ref{fig:w_{n}}). Corollary \ref{1/2} and Proposition 3.2 in  \cite{GMT18} imply that
 \begin{eqnarray}\label{5801}
\sum\limits_{A_{n}-e_{2n-1}\in \mathscr{C}_{n}} z^{\gamma[(Q_{n}-e_{n-1})^{A_{n}-e_{2n-1}}]}
 &=& ~^{\partial}{\Gamma}_{Q_{n-1}}(z).
 \end{eqnarray}

    Let  $\mathscr{A}_{n}=\{A'_{n}|  A'_{n} $\rm{\ contains\ a\ cycle\ with}$ \
 e_{2n-1}  \rm{\ in}$ $ \rm{\ Q_{n},} \ e_{1}\notin A'_{n}, \ A'_{n}\subset E(Q_{n}-e_{2n-1})\},$ by the definition of $A'_{n}$ and $A_{i}$, we know that for each $A'_{n}$, we can find a $A_{i}$ ($2\leq i \leq n-1$), such that
 $$A_{i}=A'_{n}-e_{2n-2}- e_{2n-4}-\cdots - e_{2i-2}.$$
 Note that deletion does not change the number of components of $A'_{n}$, then  $c(A'_{n})=c(A_{i})$.
Since $(A'_{n})^{c}$ has an extra isolated vertex-disk, we have  $c((A'_{n})^{c})=c(A_{i}^{c})+1$. Clearly,  $v(Q_{n})=v(Q_{i})+n-i$.  From Theorem~\ref{Gro:main}, Corollary \ref{1/2} and Theorem \ref{main}, we have
\begin{eqnarray}\label{eq:19}
f_{n}(z)&=&\sum\limits_{A_{n}\in \mathscr{C}_{n}} z^{\gamma[(Q_{n})^{A_{n}}]}\notag\\
&=&z^{\partial}{\Gamma}_{Q_{n-1}}(z)+(1-z)\sum\limits_{A'_{n}\in \mathscr{A}_{n}} z^{\gamma[(Q_{n})^{A'_{n}}]}\notag \quad\text{by  (\ref{5801})}\\
&=& z^{\partial}{\Gamma}_{Q_{n-1}}(z)+(1-z)\sum \limits_{i=2}^{n-1}z^{n-1-i}f_{i}(z).
\end{eqnarray}

   We are now in a position to calculate $\sum\limits_{A\in \mathscr{A}} z^{\gamma[W_{n}^{A}]}$. For any $A\in \mathscr{A}$, there exists $A'_{n-1}\in \mathscr{A}_{n-1}$
   such that  $A=A'_{n-1}\vee e_{2n-2}\vee e_{2n-1}$ and $A^{c}=(A'_{n-1})^{c}\vee e_{2n}.$ Note that
   $e_{2n-3}\in A'_{n-1}$, and  $e_{2n-3}$$e_{2n-2}$$e_{2n-1}$ is a 3-cycle in $A$, then
   $c(A'_{n-1})=c(A)$.
Since $e_{1} \in  (A'_{n-1})^{c}$,  ribbon $e_{1}$ and ribbon $e_{2n}$ are adjacent in $A^{c}$, we obtain     $c((A'_{n-1})^{c})=c(A^{c})$. Obviously,  $v(Q_{n-1})+1=v(W_{n})$.
The proof for $e_{2n}\in A$ and $e_{2n-1}\in A^{c}$ will not be reproduced here, since it is the same as that just given for $e_{2n-1}\in A$ and $e_{2n}\in A^{c}$. By  Theorem~\ref{Gro:main} we have
 \begin{eqnarray}\label{eq:20}
\sum\limits_{A\in \mathscr{A}} z^{\gamma[W_{n}^{A}]}&=&2\sum \limits_{i=2}^{n-1}z^{n-i}f_{i}(z).
\end{eqnarray}

It remains to calculate that $~^{\partial}{\Gamma}_{W_{n}-e_{2n-1}}(z)$. Recall that $W_{n}-e_{2n-1}$ is isomorphic to $W_{n-1}$ with ribbon $e_{2n-2}$ subdivided once,  and $W_{n-1}-e_{2n-2}$ is isomorphic to $Q_{n-1}$, as shown in Figure \ref{fig:w_{n}}.
 Theorem \ref{wn} gives
\begin{equation}\label{eq:17}
~^{\partial}{\Gamma}_{W_{n}-e_{2n-1}}(z)=~^{\partial}{\Gamma}_{W_{n-1}}(z)+2z~^{\partial}{\Gamma}_{Q_{n-1}}(z).
\end{equation}

 Using the fact that $~^{\partial}{\Gamma}_{\overline{W}_{n}^{*}}(z)=~^{\partial}{\Gamma}_{\overline{W}_{n}}(z),$   and by (\ref{E:18})- (\ref{eq:17}),  we have the following recursive relations:
\begin{eqnarray*}
~^{\partial}{\Gamma}_{W_{n}}(z)&=&z~^{\partial}{\Gamma}_{W_{n-1}}(z)+~^{\partial}{\Gamma}_{Q_{n}}(z)+2z~^{\partial}{\Gamma}_{Q_{n-1}}(z)-(2-2z)
f_{n}(z),\\
f_{n}(z)&=& z^{\partial}{\Gamma}_{Q_{n-1}}(z)+(1-z)\sum \limits_{i=2}^{n-1}z^{n-1-i}f_{i}(z),\\
~^{\partial}{\Gamma}_{Q_{n}}(z)&=&(4z+1)~^{\partial}{\Gamma}_{Q_{n-1}}(z)-4z^{2}~^{\partial}{\Gamma}_{Q_{n-2}}(z),
\quad \text{by Example \ref{thm:2m2} }
\end{eqnarray*}

\noindent with initial conditions $~^{\partial}{\Gamma}_{W_{1}}(z)=4,$ $~^{\partial}{\Gamma}_{W_{2}}(z)=4z^{2}+10z+2,$ $~^{\partial}{\Gamma}_{Q_{1}}(z)=2,$ $~^{\partial}{\Gamma}_{Q_{2}}(z)=6z+2,$
$f_{2}(z)=2z.$


\end{example}

\subsection{Ring-like planar ribbon graphs}For $n\geq 1$, let $C_n$ be a planar cycle ribbon graph,  and replace each vertex with a connected planar ribbon graph. The resulting planar ribbon graph is called a \textit{Ring-like planar ribbon graph} $R_n$. We have the following result.
\begin{thm} \label{HN}
Let  $G_1,G_2,\cdots,G_n$ be  disjoint  connected planar ribbon graphs. For $1\leq i\leq n,$ let $v_{2i}, v_{2i-1}$ be two root-vertices of $G_i,$  and let   $R_{n}= \mathop{\cup}\limits_{i=1}^{n}G_{i}\cup \mathop{\cup}\limits_{i=1}^{n-1} v_{2i}v_{2i+1}\cup v_{2n}v_{1}$ be the Ring-like planar ribbon graph.  If $\bar{G}_{i}=G_{i}+ v_{2i-1}v_{2i}$ is a planar ribbon graph for $1\leq i\leq n,$   then the pdG-polynomial of $R_n$ is
\begin{eqnarray}
~^{\partial}{\Gamma}_{R_{n}}(z)&=&2^{n}z\mathop{\prod}\limits _{i=1}^{n}~^{\partial}{\Gamma}_{G_{i}}(z)+(2-2z)\mathop{\prod}\limits_{i=1}^{n}\frac{~^{\partial}{\Gamma}_{\bar{G_{i}}}(z)
-2z~^{\partial}{\Gamma}_{G_{i}}(z)}{2-2z}.\label{HN00}
\end{eqnarray}
\end{thm}
\begin{proof}

Let $e_{i}=v_{2i}v_{2i+1}$ ($1\leq i\leq n-1$), and $e_{n}=v_{2n}v_{1}.$ For $1\leq i\leq n$, let $P_{2i-1,2i}$ be a $v_{2i-1}v_{2i}$-path in $G_i.$ Then $P=\mathop{\cup}\limits_{i=1}^{n}P_{2i-1,2i}\cup \mathop{\cup}\limits_{i=1}^{n-1}e_i$ is a path graph and $P\cup e_n$ is a cycle. Furthermore, $P_{2i-1,2i}\cup v_{2i}v_{2i-1}$ is a cycle graph in $\bar{G_{i}}$.

To prove the Theorem, we give some  definitions.
Suppose $A_{i}$ is an edge subset of $E(G_{i})$ that does not intersect with $P_{2i,2i-1}$,
 next we let $\bar{A}_{i}=P_{2i,2i-1}\cup A_{i},$ $\overline{e}_{i}=v_{2i}v_{2i-1}$ and let $A=P\cup \mathop{\cup}\limits_{i=1}^{n}A_{i}=
\mathop{\cup}\limits \limits_{i=1}^{n}\bar{A}_{i}\cup\sum \limits_{i=1}^{n-1}e_{i}$. It can be shown that the complement of $A$ in $E(R_{n}-e_{n})$ is equal to the complement of $\mathop{\cup}\limits_{i=1}^{n}\bar{A}_{i}$ in  $E(\mathop{\cup} \limits_{i=1}^{n}G_{i})$, it follows that
\begin{eqnarray*}
c(A^{c})&=&c((\mathop{\cup}\limits_{i=1}^{n}\bar{A}_{i})^{c}).
\end{eqnarray*}
 Now define
 $$\mathscr{A}=\{A| A\cup e_{n} \rm{\ contains\ a\ cycle\ with}\ e_{n}\  \rm{\ in}\ R_{n}, A\subset E(R_{n}-e_{n})\},$$ and
  $$\mathscr{A}_{i}=\{\bar{A}_{i}| \bar{A}_{i}\cup \overline{e}_{i} \rm{\ contains\ a\ cycle\ with}\ \overline{e}_{i}\  \rm{\ in}\ \bar{G_{i}},\bar{A}_{i}\subset E(G_{i})\}.$$

Note that $c(R_{n}-e_{n})=c(\mathop{\cup} \limits_{i=1}^{n} G_{i})-n+1$,  $v(R_{n}-e_{n})=v(\mathop{\cup} \limits_{i=1}^{n} G_{i})$,    and $c(A)=c(\mathop{\cup} \limits_{i=1}^{n} \bar{A}_{i})-n+1$. Therefore, according to Theorem \ref{Gro:main}, we have
$$\gamma[(R_{n}-e_{n})^{A}]=\gamma[\mathop{\cup} \limits_{i=1}^{n}G_{i}^{\bar{A}_{i}}],$$
 thus
\begin{eqnarray}
\sum\limits_{A\in \mathscr{A}} z^{\gamma[(R_{n}-e_{n})^{A}]}&=&\prod\limits_{i=1}^{n}\sum\limits_{\bar{A}_{i}\in \mathscr{A}_{i}} z^{ \gamma[G_{i}^{\bar{A}_{i}}]}.\label{HN2}
\end{eqnarray}

 By Proposition 3.2 in \cite{GMT18} and   Theorem \ref{main},
 \begin{eqnarray}
 ~^{\partial}{\Gamma}_{R_{n}-e_{n}}(z)&=&2^{n-1}\mathop{\prod}\limits _{i=1}^{n}~^{\partial}{\Gamma}_{G_{i}}(z),\label{HN5}\\
~^{\partial}{\Gamma}_{\bar{G_{i}}}(z)&=&2z~^{\partial}{\Gamma}_{G_{i}}(z)+(2-2z)\sum\limits_{\bar{A}_{i}\in \mathscr{A}_{i}} z^{\gamma[G_{i}^{\bar{A}_{i}}]}, \label{HN3}\\
~^{\partial}{\Gamma}_{R_{n}}(z)&=&2z~^{\partial}{\Gamma}_{R_{n}-e_{n}}(z)+(2-2z)\sum\limits_{A\in \mathscr{A}} z^{\gamma[(R_{n}-e_{n})^{A}]}.\label{HN4}
\end{eqnarray}
 The result follows by substituting (\ref{HN2}), (\ref{HN5}) and  (\ref{HN3}) into  (\ref{HN4}).

\end{proof}
\begin{example}
A ribbon graph consisting of $n$ multiple ribbons with two vertices is called \textit{dipole ribbon graph} $D_{n}$. A \textit{necklace graph}  is obtained by adding $n$ disjoint edges to the cycle graph $C_{2n}$.  Let $N_n=\mathop{\cup}\limits_{i=1}^{n}G_{i}\cup \mathop{\cup}\limits_{i=1}^{n-1} v_{2i}v_{2i+1}\cup v_{2n}v_{1}$  be the planar ribbon necklace graph as shown  in Figure \ref{fig:G_{6}}. For $1\leq i\leq n,$ let $G_i$ be a copy of $D_2$ with vertices $v_{2i}, v_{2i-1}$. By Theorem \ref{HN}, we know $\bar{G_{i}}$ is $D_{3}$. Clearly, $~^{\partial}{\Gamma}_{D_{k}}(z)=2+(2^{k}-2)z.$
\begin{eqnarray}\label{necklace:formula}
~^{\partial}{\Gamma}_{N_{n}}(z)&=&2^{n}z\mathop{\prod}\limits _{i=1}^{n}~^{\partial}{\Gamma}_{D_{2}}(z)+(2-2z)(\frac{~^{\partial}{\Gamma}_{D_{3}}(z)-
2z~^{\partial}{\Gamma}_{D_{2}}(z)}{2-2z})^{n}\quad \text{by (\ref{HN00})}\notag\\
&=&2^{n}z(2+2z)^{n}+(2-2z)(1+2z)^{n}.
\end{eqnarray}

\begin{figure}[h]
  \centering
  \includegraphics[width=0.45\textwidth]{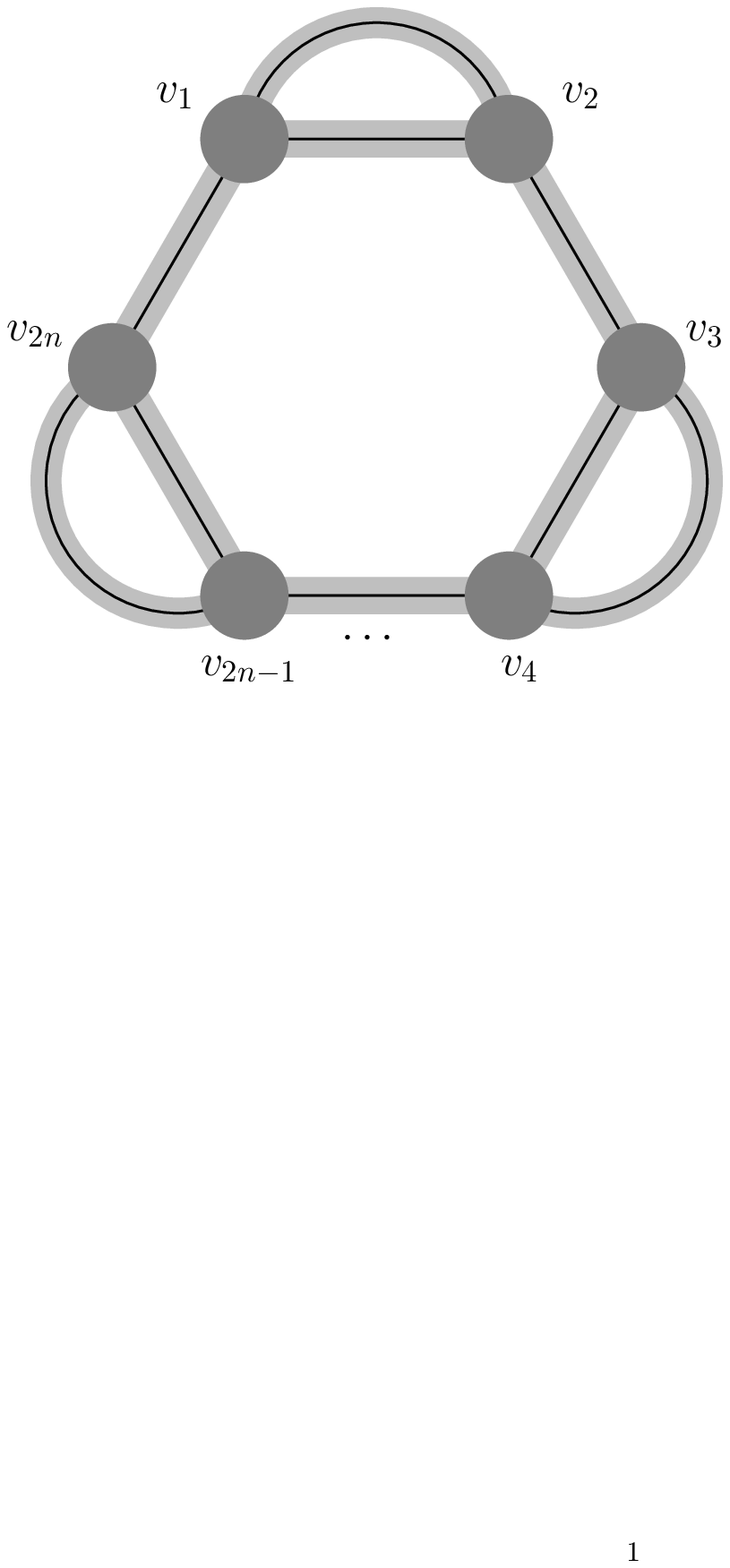}
 \caption{$N_{n}$}
 \label{fig:G_{6}}
\end{figure}
\end{example}


\section{Asymptotic partial-dual genus distribution}

In this section we present asymptotic results for partial-dual genus distributions of some infinite ribbon graph families. 

\begin{thm}The partial-dual genus distribution for the planar standard fan graph $F_{(2,n,2)}$ is asymptotic to normal distribution with mean $\frac{2(n+3)}{3}$ and variance $4n+12$ when $n$ tends to infinity.
\end{thm}
\begin{proof}By (\ref{equ:add:Qn}),  $~^{\partial}{\Gamma}_{F_{(2,n,2)}}(z)=~^{\partial}{\Gamma}_{Q_{n+3}}(z).$ Thus, it suffices to prove that the partial-dual genus distribution of $Q_n$ is asymptotically normal distribution when $n$ tends to infinity. Since \begin{equation*}
 ~^{\partial}{\Gamma}_{Q_{n}}(z)=(4z+1)~^{\partial}{\Gamma}_{Q_{n-1}}(z)-4z^{2}~^{\partial}{\Gamma}_{Q_{n-2}}(z),
\end{equation*}
we define its  characteristic equation as $F(z,\lambda)=\lambda^2-(4z+1)\lambda+4z^2.$ Then the two solutions of $F(z,\lambda)$ are $\lambda_1(z)=\frac{4z+1+\sqrt{8z+1}}{2}$ and $\lambda_2(z)=\frac{4z+1-\sqrt{8z+1}}{2}.$ One easily see that $|\lambda_1(z)|>|\lambda_2(z)|,$ for $z\in R$ and we have  \begin{eqnarray*}
\label{ppp-1}
  e=\frac{\lambda_1'(1)}{\lambda_1(1)}=\frac{2}{3},\quad  v=\frac{-\big(\lambda_1'(1)\big)^2 +\lambda_1(1)\cdot  \lambda_1^{''}(1)
 +\lambda_1(1)\cdot\lambda_1'(1) }{\lambda_1(1)^2}=4>0.
\end{eqnarray*}
From Theorem 2.1 in \cite{ZPC19}, the partial-dual genus distribution of $Q_n$ is asymptotically normal distribution with mean $\frac{2n}{3}$ and variance $4n$ when $n$ tends to infinity. The proof is completed.
\end{proof}

\begin{thm}
Consider  the planar ribbon necklace graph sequence   $\{N_{n}\}_{n=1}^\infty$. Then, the partial-dual genus distribution   of  $N_n$  is  asymptotically normal distribution with mean  $\frac{n+2}{2}$ and variance $\frac{n}{4}$ when $n$ tends to infinity.
\end{thm}

\begin{proof}We first compute the mean and variance for the partial-dual genus distribution  of $N_n.$ For any $n\in {N},$ let $X_{N_n}$ ($X_{n}$, for short) be a random variable with distribution
\begin{eqnarray}
\label{b-1}
  p_i=\mP(X_{n}=i)&=&\frac{\gamma_i(N_{n})}{2^{e(N_n)}},\quad i=0,1,2,\cdots,.
\end{eqnarray}  By (\ref{necklace:formula}), $~^{\partial}{\Gamma}_{N_{n}}(z)=2^{n}z(2+2z)^{n}+(2-2z)(2z+1)^{n},$ we get

$$
\begin{aligned}
P_{X_{n}}'(z)&=\frac{2^{2n}((n+1)z+1)(1+z)^{n-1}+2(-(2n+2)z+2n-1)(2z+1)^{n-1}}{2^{3n}}, {\; and}
\\
P_{X_{n}}''(z)&=\frac{2^{2n}(1+z)^{n-2}((n+n^2)z+2n)+2((-4n^2-4n)z+4n^2-8n)(2z+1)^{n-2}}{2^{3n}}.
\end{aligned}
$$

Thus

$$
\begin{aligned}
\mathbb{E}(X_n)&=P_{X_{n}}'(1)=\frac{(n+2)2^{3n-1}-2\times3^{n}}{2^{3n}}=\frac{n+2}{2}-2(\frac{3}{8})^n,\\
\mathbb{E}(X_n^2-X_n)&=P_{X_{n}}''(1)=\frac{(n^2+3n)2^{3n-2}-24n\times3^{n-2}}{2^{3n}}=\frac{n^2+3n}{4}-\frac{8n}{3}(\frac{3}{8})^n.
\end{aligned}
$$

After straightforward  calculations,  one has the following equation

\begin{eqnarray*}
   Var(X_{n})&=& P_{X_{n}}''(1)+P_{X_{n}}'(1)- (P_{X_{n}}'(1))^2=\frac{n}{4}+(-\frac{2n}{3}+2)(\frac{3}{8})^n-4(\frac{3}{8})^{2n}.
\end{eqnarray*}

Suppose $\phi_{n}(t)$ is the moment generating function of $  \frac{X_{n}- \frac{n+2}{2} }{\sqrt{\frac{n}{4}}}.$ I.e.,

\begin{eqnarray}
 \phi_{n}(t)=\mathbb{E} \exp{ {t\frac{X_{n}- \frac{n+2}{2}}{\sqrt{\frac{n}{4}}}}}=\mathbb{E} \exp\{a_{n} t X_{n}-b_{n}t\},
 \end{eqnarray}
where
\begin{eqnarray}
  \label{5-1} a_{n}=\frac{1}{\sqrt{\frac{n}{4}}},\quad b_{n}=\frac{\frac{n+2}{2}}{\sqrt{\frac{n}{4}}}.
\end{eqnarray}
Let $z=z(n,t)=e^{a_nt},$ then we have

\begin{eqnarray*}
 \phi_{n}(t)&=&\mathbb{E} \exp\{a_{n} t X_{n}-b_{n}t\}\\&=&e^{-b_{n}t} P_{X_{n}}(z)\\
 &=&e^{-b_{n}t}\frac{2^{n}z(2+2z)^{n}+(2-2z)(2z+1)^{n}}{2^{3n}}\\
 &=&e^{-b_{n}t}[\frac{1}{2^n}z(1+z)^n+2^{1-3n}(1-z)(2z+1)^n].
 \end{eqnarray*}
Noting that $z=1$ when $n$ tends to infinity. Thus
\begin{eqnarray*}
   \lim_{\substack{n\rightarrow \infty  }} \phi_{n}(t)&=&\lim_{\substack{n\rightarrow \infty }} e^{-b_{n}t}\left[\frac{1}{2^n}z(1+z)^n+2^{1-3n}(1-z)(2z+1)^n\right]\\
   &=&\lim_{\substack{n\rightarrow \infty }} e^{-b_{n}t}\left[\frac{1}{2^n}z(1+z)^n\right]\\
   &=&\lim_{\substack{n\rightarrow \infty }} e^{-b_{n}t}e^{\ln (z(\frac{1+z}{2})^n)}\\
   &=&\lim_{\substack{n\rightarrow \infty }}\exp{(-b_{n}t+\ln z+n\ln\frac{1+z}{2})}.
   \end{eqnarray*}

From Taylor formula,

\begin{eqnarray*}
 \ln\frac{1+z}{2}&=&  \frac{1}{2}(z-1)-\frac{1}{8} (z-1)^2+O\left((z-1)^3\right),
\\  \ln z&=&(z-1)-\frac{1}{2}(z-1)^2+O\left((z-1)^3\right),{\; and}\\
\\ z&=&e^{a_{n}  t}=1+a_{n}t+\frac{1}{2}a_{n}^2t^2+O\left(a_{n}^3t^3\right)
\end{eqnarray*}
Thus,

\begin{eqnarray*}
 \ln\frac{1+z}{2}&=&  \frac{1}{2}(a_{n}t+\frac{1}{2}a_{n}^2t^2)-\frac{1}{8} (a_{n}t+\frac{1}{2}a_{n}^2t^2)^2+O\left((a_{n}t+\frac{1}{2}a_{n}^2t^2)^3\right)\\
 &=&\frac{1}{2}a_{n}t+\frac{1}{8}a_{n}^2t^2 +O(a_{n}^3t^3).
\\ \ln z&=&a_{n}t+O(a_{n}^3t^3).
\end{eqnarray*}

Recall that $a_n=\frac{2}{\sqrt{n}}$ and $b_n=\frac{n+2}{\sqrt{n}},$ we have
\begin{eqnarray*}
   \lim_{\substack{n\rightarrow \infty  }} \phi_{n}(t)&=&\lim_{\substack{n\rightarrow \infty }}\exp{(-b_{n}t+\ln z+n\ln\frac{1+z}{2})}\\
   &=&\lim_{\substack{n\rightarrow \infty  }}\exp {(-\frac{n+2}{\sqrt{n}}t+\frac{2}{\sqrt{n}}t+n(\frac{1}{\sqrt{n}}t+\frac{1}{2n}t^2))}\\
   &=&\lim_{\substack{n\rightarrow \infty  }}\exp {\frac{1}{2}t^2}\\
   &=&e^{\frac{t^2}{2}},
   \end{eqnarray*}

which finishes the proof of this theorem.
\end{proof}

Note that the necklace $N_n$ above is a cubic outerplanar graph. Let $O_n$ be any cubic outerplanar ribbon graph, in \cite{Gro11}, Gross proved that any cubic outerplanar graph can be decomposed into star-ladder graph families. By Theorem \ref{HN}, we can calculate its pdG-polynomial recursively. A natural question is:
\begin{problem}
Are the partial-dual genus distributions of $O_n$  asymptotically normal  when $n$ tends to infinity?
\end{problem}

We also think it's a universal phenomenon that a Gaussian limit law should hold.

\subsection{Conclusions}
Recall that some relations between planar graphs and partial dual genus are also given in  \cite{Mof12}. Suppose $G$ is a ribbon graph and $A\subset E(G).$ It is  clear that  $~^{\partial}{\Gamma}_{G}(z)=~^{\partial}{\Gamma}_{G^{A}}(z)$ \cite{YJ21}. For a non-planar ribbon graph, if  $G^{A}$ is a planar ribbon graph for some $A\subset E(G),$ then we also can apply Theorem \ref{main} to enumerate its pdG-polynomial. In \cite{Mof12}, Moffatt  gives a characterization of a non-planar ribbon graph $G$ such that $G^{A}$ is planar. This is one of motivations for us to look at the partial-duals of planar ribbon graphs. For the non-planar ribbon graphs, we will discuss this in another paper \cite{CC21}.

\subsection{Acknowledgments}

This work is supported by the NNSFC (Grant No. 11471106) and the JSSCRC( Grant No.
2021530).
We are grateful to the anonymous referee for the valuable comment.

\end{document}